\documentclass[11pt,a4paper]{amsart}
\usepackage[utf8]{inputenc}

\usepackage[
    style=alphabetic,
    maxnames=6,
    backend=biber,
    isbn=false,
    doi=false,
    eprint=true,
    url=false]
    {biblatex}
\usepackage[
    colorlinks=true,
    citecolor=magenta,
    urlcolor=blue
    ]{hyperref}
\usepackage{mathtools}
\usepackage{amsthm}
\usepackage{amssymb}
\usepackage{bm}
\usepackage{scalerel}
\usepackage{thmtools}
\usepackage{enumitem}
\usepackage{subcaption}

\usepackage{xurl}
\hypersetup{breaklinks=true} %
\usepackage{lmodern}
\usepackage{microtype}

\usepackage{tikz}
\usetikzlibrary[matrix]
\usetikzlibrary[calc]
\usetikzlibrary[tikzmark]

\usepackage{asymptote}
\definecolor{asypurple}{HTML}{8000FF}

\newcommand{\maxtropline}[3]{%
    \filldraw (#1, #2) circle (0.03);
    \draw (0, #2) -| (#1, 0);
    \draw (#1, #2) -- ($(0,#3)!(#1, #2)!(#3,0)$); %
}
\newcommand{\maxtroplinecolor}[4]{%
    \filldraw (#1, #2) [#4] circle (0.03);
    \draw (0, #2) [#4] -| (#1, 0);
    \draw (#1, #2) [#4] -- ($(0,#3)!(#1, #2)!(#3,0)$); %
}

\theoremstyle{definition}
\declaretheorem[within=section]{definition}
\declaretheorem[within=section,sibling=definition,qed=$\diamondsuit$]{example}
\declaretheorem[name={Example}, numbered=no,within=section,sibling=definition,qed=$\diamondsuit$]{example*}

\theoremstyle{plain}
\declaretheorem[within=section,sibling=definition]{corollary}
\declaretheorem[within=section,sibling=definition]{lemma}
\declaretheorem[within=section,sibling=definition]{proposition}
\declaretheorem[within=section,sibling=definition]{theorem}
\declaretheorem[name={Theorem}]{maintheorem}
    
\declaretheorem[name={Corollary},sibling=maintheorem]{maincorollary}

\theoremstyle{remark}
\declaretheorem[within=section,sibling=definition]{remark}

\newcommand{\C}{\mathbf{C}}
\newcommand{\F}{\mathbf{F}}
\newcommand{\G}{\mathbf{G}}
\newcommand{\K}{\mathbf{K}}
\newcommand{\N}{\mathbf{N}}
\renewcommand{\P}{\mathbf{P}}
\newcommand{\Q}{\mathbf{Q}}
\newcommand{\T}{\mathbf{T}}
\newcommand{\TR}{\mathbf{T}\!\!\mathbf{R}}
\newcommand{\R}{\mathbf{R}}
\renewcommand{\S}{\mathbf{S}}
\newcommand{\Z}{\mathbf{Z}}

\newcommand{\PF}{{\mathrm{PF}}}

\newcommand{\initial}{\mathrm{in}}
\newcommand{\result}[1]{R_{ #1}}

\newcommand{\cF}{\mathcal{F}}
\newcommand{\cL}{\mathcal{L}}

\DeclareMathOperator{\mult}{mult}
\DeclareMathOperator{\bmult}{\partial-mult}
\DeclareMathOperator{\gmult}{gmult}
\DeclareMathOperator{\pmult}{\epsilon-mult}
\newcommand{\sN}{{\epsilon\text{-}N}}

\DeclareMathOperator{\val}{val}
\newcommand{\sign}{{\mathrm{sgn}}}
\newcommand{\sgn}{{\mathrm{sgn}}}
\newcommand{\supp}{{\mathrm{supp}}}
\newcommand{\angular}{{\mathrm{ac}}}
\newcommand{\vangular}{{\nu_\mathrm{ac}}}
\newcommand{\vsign}{{\nu_\sign}}
\newcommand{\Newt}{{\mathrm{Newt}}}

\newcommand{\mc}[1]{\mathcal{#1}}

\DeclareMathOperator*{\bigboxplus}{\raisebox{-0.5em}{\scaleobj{2}{\boxplus}}}

\author{Andreas Gross}
\author{Trevor Gunn}

\title[The multivariate Descartes' problem]{Factoring multivariate polynomials over hyperfields and the multivariable Descartes' problem}
\date{\today}

\begin{document}
\begin{abstract}
    We develop several notions of multiplicity for linear factors of multivariable polynomials over different arithmetics (hyperfields).
    The key example is multiplicities over the hyperfield of signs, which encapsulates the arithmetic of $\mathbf{R}/\mathbf{R}_{>0}$.
    These multiplicities give us various upper and lower bounds on the number of linear factors with a given sign pattern in terms of the signs of the coefficients of the factored polynomial. Using resultants, we can transform a square system of polynomials into a single polynomial whose multiplicities give us bounds on the number of positive solutions to the system.
    In particular, we are able to re-derive the lower bound of Itenberg and Roy on any potential upper bound for the number of solutions to a system of equations with a given sign pattern. In addition, our techniques also explain a particular counterexample of Li and Wang to Itenberg and Roy's proposed upper bound. 
\end{abstract}
\maketitle

\section*{Introduction}

\subsection*{Background}

Famously, Descartes' Rule of Signs states that the number of positive solutions of a polynomial
\[
    f(x) = a_0 + a_1 x + \dots + a_n x^n \in \R[x]
\]
is bounded above by the number of sign changes of the sequence of coefficients $a_0,\ldots, a_n$. Numerous proofs have been found since Descartes' original work \cite{Proof1Descartes,Proof2Descartes}, some of which are extremely short \cite{ShortProof1,ShortProof2}. There are several generalizations of Descartes' Rule of Signs as well: the Budan–Fourier theorem and Sturm's theorem give estimates of the number of solutions of real polynomials in a given interval in terms of the number of sign changes of suitable sequences of real numbers. Laguerre proved, using Rolle's theorem, that Descartes' rule also holds if the exponents appearing in $f$ are arbitrary real numbers, and the problem of finding and characterizing more general functions satisfying Descartes' rule has received some attention \cite{Generaliztion1Descartes, Generalization2Descartes, CurtissPowerSeries}.
Descartes' bound (in the polynomial setting) is also known to be sharp \cite{SharpnessDescartes}.

In multiple variables, one possible generalization of Descartes' rule considers a single polynomial $f(\bm x)$ in several variables and  asks on how many components of the complement of its vanishing set the polynomial $f(\bm x)$ can be positive, given the signs of its coefficients \cite{DescartesHypersurfaces}. Another generalization considers systems of real polynomial equations $0=f_1(\bm x)=f_2(\bm x)= \cdots$ and asks how many solutions with only positive entries such a system can have, given the signs of the coefficients of each of the $f_i$. This latter formulation was first studied by Itenberg and Roy \cite{IR}, who made a conjecture for a sharp upper bound of positive solutions in terms of Newton polytopes and mixed subdivisions. Popularized by a \$500 bounty by Bernd Sturmfels, the conjecture received some attention and was later disproven \cite{LW}. More recently, Bihan-Dickenstein and Bihan-Dickenstein-Forsgård gave a sharp upper bound for the number of positive solutions of systems of polynomials supported on a circuit \cite{BD, BDF}. The general case is still wide open.

\begin{example*} \label{ex:resultant-might-change}
    With multiple variables, it is possible to have a family of equations with consistent signs but whose solutions have varying signs. This phenomenon does not happen in one variable where, if the coefficients change $k$ times, Descartes' rule tells us that there will always be exactly $k$ positive roots assuming all the roots are real. For example, consider the system
    \begin{align*}
        x^2 + y^2 &= 1, \\
        ax + by &= 1, \\
        a, b &> 0.
    \end{align*}
    The space of real solution sets consists of four open components as shown in Figure~\ref{fig:sign-patterns}.
\end{example*}

\begin{figure}[htbp]
    \centering
    \begin{tikzpicture}
        \draw[dashed] (-1.1, 0) -- (1.1, 0);
        \draw[dashed] (0, -1.1) -- (0, 1.1);
        \draw (0, 0) circle (1);
        \draw (-1.1, 0.8) -- (1.1, 0.2);
    \end{tikzpicture}
    \quad
    \begin{tikzpicture}
        \draw[dashed] (-1.1, 0) -- (1.1, 0);
        \draw[dashed] (0, -1.1) -- (0, 1.1);
        \draw (0, 0) circle (1);
        \draw (-1.1, 0.8) -- (1.1, -0.4);
    \end{tikzpicture}
    \quad
    \begin{tikzpicture}
        \draw[dashed] (-1.1, 0) -- (1.1, 0);
        \draw[dashed] (0, -1.1) -- (0, 1.1);
        \draw (0, 0) circle (1);
        \draw (0.1, 1.1) -- (1.1, 0.1);
    \end{tikzpicture}
    \quad
    \begin{tikzpicture}
        \draw[dashed] (-1.1, 0) -- (1.1, 0);
        \draw[dashed] (0, -1.1) -- (0, 1.1);
        \draw (0, 0) circle (1);
        \draw (0.8, -1.1) -- (0.2, 1.1);
    \end{tikzpicture}
    \caption{Possible sign patterns which arise from intersecting a line with the unit circle.}
    \label{fig:sign-patterns}
\end{figure}
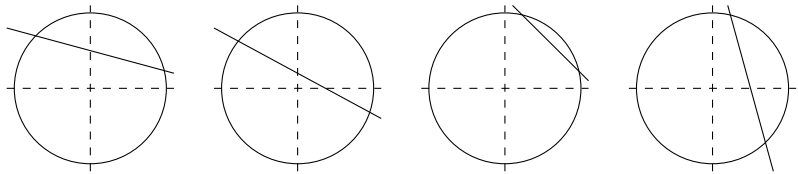

\subsection*{Descartes' rule and hyperfields}
Hyperfields are generalizations of fields, where addition may be multivalued. These appear naturally when looking at the quotient of a field by a multiplicative group. For instance, we can take the real numbers and quotient by the group of absolute values ($\R_{>0}$) to obtain the hyperfield of signs $\S = \{+1, -1, 0\}$. The arithmetic of signs has rules such as $1+1=1$ (the sum of two positive numbers is always positive) and $1+(-1)=\S$ (the sum of a positive and negative number may have any sign). Similarly, if we quotient $\R$ by $\{\pm 1\}$, we get a hyperfield which encapsulates the arithmetic of absolute values. Arithmetic of non-Archimedean absolute values is often used in tropical geometry. We call this hyperfield the \emph{tropical hyperfield}, $\T$. This hyperfield is an enrichment of the tropical semifield. We can also combine signs and non-Archimedean absolute values with the so-called real tropical hyperfield $\TR$, which is a sort of semidirect product of $\S$ and $\T$. This hyperfield is useful to describe real tropical geometry \cite{JSY}.

In their recent paper \cite{BL}, Baker and Lorscheid have given a proof of Descartes' Rule of Signs using hyperfields. What they show is that given a real polynomial $f(x)\in \R[x]$ with $n$ positive roots, its image $f^\sgn$ in $\S[x]$ must be divisible by $x-1 \in \S[x]$ at least $n$ times. The multiplicity $\mult^\S_{x-1}(f^\sgn)$ of $x - 1$ as a factor of $f^\sgn$ therefore bounds the number of positive roots of $f$ from above. 
Moreover, Baker and Lorscheid show that the maximal number of times one can factor out $x - 1$ (i.e.\ $\mult_{x - 1}^\S (f^\sgn)$)
is exactly the number of sign alterations as in  Descartes' rule. Their theory also applies to the tropical hyperfield \cite{BL} as well as other hyperfields like those associated to higher rank valuations or combining valuations and signs \cite{G, G2}. Akian-Gaubert-Tavikalipour have also carried out similar factorization results for polynomials over Rowan's ``semiring systems'' \cite{AGT}.

\subsection*{Linear factors of multivariate polynomials}

An analogous formulation of Descartes' rule that has, so far, received little attention asks the following: given a polynomial $f(\bm x)$ in \emph{several} variables with given support and coefficients with prescribed signs, what is the sharp upper bound for the number of its linear factors with a prescribed sign pattern? There is some relationship between this problem and the system-of-equation problem because the sparse resultant of a system of equations yields a single polynomial whose linear factors correspond (with multiplicity!) to the common solutions of the system. However, as shown in the example above, the signs of the resultant are not uniquely determined from the signs of the system.

We approach the linear factor problem with the same strategy used by Baker and Lorscheid \cite{BL} in the univariate case: for a real multivariate polynomial $f(\bm x)\in \R[\bm x]$ and a ``signed'' degree-$1$ polynomial $l=s_0 +\sum s_i x_i\in \S[\bm x]$, we define $\mult^\R_{\sgn^{-1}\{l\}}(f)$ as the maximal number of degree-$1$ polynomials $k$ with $k^\sgn=l$ that we can factor out of $f$. Similarly, we define $\mult^\S_l(f^\sgn)$ as the maximal number of times that we can factor $l$ out of $f^\sgn$ (as pointed out by Baker and Lorscheid \cite{BL}, one has to be careful here since quotients are not unique; see Definition~\ref{def:hyperfield multiplicity}).

\begin{maintheorem} [= Lemma~\ref{lem:multiplicities and morphisms}]
\label{intro:bounding number of factors}
    We have 
    \[
    \mult^\R_{\sgn^{-1}\{l\}}(f) = \sum_{k} \mult^\R_k(f) \leq \mult^\S_l(f^\sgn),
    \]
    where we sum over a set of representatives $k$ of the image of $\sign^{-1}\{l\}$ in $\R[\bm x]/\R^*$, using unique factorization in $\R[\bm x]$.
\end{maintheorem}

Even in the one variable case, a real polynomial might have complex roots, meaning its observed number of positive roots could be less than the maximum allowed by its sign configuration.
We define the relative multiplicity (with respect to $\sgn$) of $l$ in a polynomial $g\in \S[\bm x]$, by
\[
\mult^\sgn_l(g)= \max\{\mult^\R_{\sgn^{-1}\{l\}}(f) : f^\sgn=g\}.
\]
Then the problem of finding the sharp upper bound for the number of linear factors with prescribed sign pattern in a polynomial with coefficients of prescribed signs becomes the question of determining the relative multiplicities $\mult^\sgn_l(g)$. As an immediate consequence of the Theorem \ref{intro:bounding number of factors}, we obtain the following corollary.

\begin{maincorollary}[= Proposition~\ref{prop:comparison relative multiplicity and hyperfield multiplicity}]
\label{intro:bound not sharp}
    For $l\in \S[\bm x]$ of degree $1$ and $g\in \S[\bm x]$ arbitrary we have
    \[
    \mult^\sgn_l(g)\leq \mult^\S_l(g) .
    \]
\end{maincorollary}

Note that we prove Corollary~\ref{intro:bound not sharp} in much greater generality, where $\sgn$ is replaced by an arbitrary morphism of hyperfields.

\begin{example*}
    Let
    \[ f = (x - 1)(x - 2)(x^2 + 2) = x^4 - 3x^3 + 4x^2 - 6x + 4 \in \R[x]. \]
    Then $f^\sign = x^4 - x^3 + x^2 - x + 1 \in \S[x]$. By Descartes' rule~\cite[Theorem~C]{BL}, we have $\mult^\S_{x - 1}(f^\sign) = 4$ (the number of sign changes) but
    \[
        \mult^\R_{\sign^{-1}\{x - 1\}}(f) = \mult^\R_{x - 1} f + \mult^\R_{x - 2} f = 2.
    \]
    On the other hand, $\mult^\sign_{x - 1}(f^\sign) = 4$ since, for example, $(x - 1)^4$ is a real polynomial in $\sign^{-1}\{f^\sign\}$ with $4$ positive roots.
\end{example*}

The sharpness in Descartes' rule of signs for univariate polynomials means precisely that $\mult^\sign_l(g) = \mult^\S_l(g)$ for any $g \in \S[x]$. 
In more than one variable, this is not true.

\begin{maintheorem}[= Example~\ref{ex:hyperfield multiplicity larger than relative multiplicity}]
    There exists a degree-$3$ polynomial $g\in \S[x,y]$ and a degree-$1$ polynomial $l\in \S[x,y]$ with
    \[
     \mult^\sgn_l(g) < \mult^\S_l(g) .
    \]
\end{maintheorem}

In addition to not being a sharp bound for the relative multiplicity, we do not have a combinatorial description for the multiplicity $\mult^\S_l(g)$ like in the univariate case. This makes the multiplicity hard to compute. In practice, it is often sufficient to work with what we call the boundary multiplicity $\bmult^\S_l(g)$, which is the maximum of the multiplicities obtained after setting one of the variables to $0$.

\subsection*{Subdivisions, Geometry and Multiplicities}
Something that makes factoring tropical polynomials easier than factoring sign polynomials is that there is a geometry associated to tropical polynomials. A linear factor of a tropical polynomial corresponds to a tropical hyperplane within the tropical hypersurface defined by that polynomial. For a polynomial over $\TR$, we define enriched tropical hypersurfaces and consider the multiplicities of enriched linear hyperplanes. We call this the \emph{(enriched) geometric multiplicity}. See Figure~\ref{fig:signed-subdivision} for a demonstration of this idea.

Looking the opposite way, if we have a polynomial over $\S$, then we can try to perturb the coefficients a little bit to yield a polynomial over $\TR$. Where the geometric multiplicity tells us to exploit an existing subdivision of the Newton polytope, here we \emph{impose} a subdivision by perturbing coefficients. We call this the \emph{perturbation multiplicity}, $\pmult^\S_l(g)$. The perturbation multiplicity is a lower bound on the hyperfield multiplicity because factoring with respect to an imposed subdivision is stricter than factoring irrespective of a subdivision. Moreover, it is also a lower bound on the relative multiplicity because the factors with the imposed subdivision can be lifted to, say, the real Puiseux series.

\begin{maintheorem}[= Corollary~\ref{cor:comparison of multiplicity with perturbation multiplicity}, Proposition~\ref{prop:comparison relative multiplicity and hyperfield multiplicity}, Corollary~\ref{cor:mult leq bmult}, Theorem~\ref{thm:multiplicity for dense quadratic poly in 2 variables}]
    If $f\in \S[\bm x]$ is dense---meaning every monomial of degree $\le \deg f$ has a non-zero coefficient---and $l\in \S[\bm x]$ has degree $1$, then we have
    \[
        \pmult^\S_l(f) \le  \mult^\sign_l(f)  \le \mult^\S_p(f) \le \bmult^\S_p(f).
    \]
    If $f$ is dense of degree two in two variables, then we have equality everywhere.
\end{maintheorem}

\subsection*{Systems of equations}

Let $\varphi\colon K\to H$ be a morphism from a field $K$ to a hyperfield $H$. 
Given polynomials $g_1,\ldots, g_n\in H[x_1,\ldots,x_n]$ and $\bm h \in (H^*)^n$ we denote by
\[
N^\varphi_{\bm h}(g_1,\ldots, g_n)
\]
the maximal number of solutions $\bm x$ with $\varphi(\bm x) = \bm h$ that a system $f_1(\bm x)=\dots =f_n(\bm x)=0$ of equations over $K$ with finite solution set (in $\overline K$) and $f_i^\sgn=g_i$ can have. For $K=\C$ and $H=\K$, the answer is given by the Bernstein-Khovanskii-Kushnirenko (BKK) theorem. For $\varphi=\sgn \colon \R \to \S$ these are precisely the numbers studied by Itenberg and Roy \cite{IR}. Let $f_i\in K[\bm x]$ with $f_i^\varphi=g_i$. Introducing an auxiliary linear form $l = 1 + y_1x_1 \dots y_nx_n$ with indeterminate coefficients and taking the (mixed sparse) resultant $\result{f_1,\ldots, f_n}\in K[\bm y]$ of $f_1,\ldots, f_n, l$, finding solutions to the system of equations 
\[
f_1(\bm x)= \dots =f_n(\bm x)=0
\]
is equivalent to finding linear factors of $R$. More precisely, if the coefficients of $f_1, \dots, f_n$ are generic, then we have 
\[
    \result{f_1,\dots,f_n} \propto \prod_{\bm a \in V(f_i) \subset (\overline K^*)^n} (1 + a_1y_1 + \dots + a_ny_n),
\]
with the proportionality being up to a unit.
The  polynomial $\result{f_1,\dots,f_n}$ is a specialization of a polynomial $\result{A_1,\dots,A_n} \in \Z[\bm y]$ which is determined just by the support sets $A_i = \supp(f_i)$.
Resultants allow us to apply our techniques to systems of equations:

\begin{maintheorem}[=Theorem~\ref{thm:upper bound by resultant}]
    Let $\result{g_1,\ldots, g_n}\subseteq H[\bm y]$ be the set of polynomials obtained by evaluating the resultant $\widetilde R_{A_1,\dots,A_n} \in \Z[\bm y]$ at the coefficients of the $g_i$, where $A_i = \supp(g_i)$. Moreover, let $l_{\bm h}= 1 + \sum h_iy_i$. Then we have
    \[
    N^\varphi_{\bm h}(g_1,\ldots, g_n) \leq \max\{ \mult^H_{l_{\bm h}}(r) : r\in \result{g_1,\ldots, g_n}\}.
    \]
\end{maintheorem}

We observe in several examples that the bound is far from sharp. However, applying the theorem to the counterexample to the Itenberg-Roy conjecture given by Li and Wang \cite{LW} yields the correct bound and shows that Li and Wang have in fact chosen an example where the number of positive solutions is maximal for the given choices of supports and signs.

We also study the numbers $N^\varphi_{\bm h}(g_1,\ldots,g_n)$ when $\varphi$ is a valuation and $H=\T$ or $H=\TR$, depending on whether $K$ is algebraically closed or real closed. In this case each of the $g_i$ defines a tropical hypersurface $V(g_i)$ and we study the case where the intersection $\bigcap_{i=1}^n V(g_i)$ is transverse at the image of $\bm h$ in $\R^n$ (this means that if $H=\TR$ we apply the projection $\TR\to \T$ coordinate-wise). Using a result by Sturmfels on initial forms of resultants \cite{sturmfels-resultants}, we prove the following result.

\begin{maintheorem}[= Theorem~\ref{thm:intersection multiplicity in transversal case}]
\label{intro:systems of equations in transverse case}
    Assume that $H=\T$, that $\varphi$ is a valuation, and that $\bigcap_{i=1}^n V(g_i)$ meets transversely at $\bm h$. Then $N_{\bm h}^\varphi(g_1,\ldots,g_n)$ equals the multiplicity of the tropical intersection product $V(g_1)\cdots V(g_n)$ at $\bm h$. If $H=\TR$ and $\varphi$ is the ``signed valuation'', then $N^\varphi_{\bm h}(g_1,\ldots,g_n)$ equals $1$ if $\bm h$ is an alternating point of $V(g_1)\cdots V(g_n)$ and $0$ otherwise (see page \pageref{def:alternating} for a definition of alternating).
\end{maintheorem}

Combining Theorem~\ref{intro:systems of equations in transverse case} with the completeness of the theory of real closed fields, we obtain a combinatorial multiplicity $\sN^\sgn_{\bm h}(g_1,\ldots,g_n)$ in terms of transverse tropical intersections or, dually, mixed Newton subdivisions. It is analogous to the combinatorial multiplicities $\pmult_l(g)$ and agrees with the numbers appearing in the conjecture of Itenberg and Roy. Our methods allow us to reprove Itenberg and Roy's lower bound.

\begin{maincorollary}[{\cite{IR}}, Corollary~\ref{cor:Itenberg roy lower bound}]
    For $g_1,\ldots, g_n\in \S[x_1,\ldots, x_n]$ and $h\in (\S^*)^n$ we have
    \[
    \sN^\sign_{\bm h}(g_1,\ldots, g_n)\leq N^\sgn_{\bm h}(g_1,\ldots, g_n) .
    \]
\end{maincorollary}

\subsection*{Acknowledgement} We thank Matt Baker and Josephine Yu for numerous insightful discussions. We thank Matt Baker and Oliver Lorscheid for providing comments on an earlier draft.

This project has received funding from the Deutsche Forschungsgemeinschaft (DFG, German Research Foundation) TRR 326 \emph{Geometry and Arithmetic of Uniformized Structures}, project number 444845124; and \emph{From Riemann surfaces to tropical curves (and back again)}, project number 456557832.

\subsection*{Notation} \hfill

\noindent
\begin{minipage}{\linewidth}
\begin{center} \textbf{Hyperfields} \end{center}
\begin{tabular}{ll}
     $\K$ & Krasner hyperfield \ref{ex:krasner-hyperfield} \\
     $\S$ & Sign hyperfield \ref{ex:sign-hyperfield} \\
     $\T$ & Tropical hyperfield \ref{ex:tropical-hyperfield} \\
     $H \rtimes \Gamma, \TR$ & Tropical extensions, tropical real hyperfield \ref{ex:tropical-extension} \\
     $ht^w = (h, w)$ & Element of a tropical extension \\
\end{tabular}
\vspace{1 em}
\end{minipage}

\noindent
\begin{minipage}{\linewidth}
\begin{center} \textbf{Maps and Morphisms} \end{center}
\begin{tabular}{ll}
    $\sgn\colon K \to \S$ & The sign of an element of a real field \ref{def:real-field} \\
    $\nu\colon K \to \T$ & A (Krull) valuation \ref{def:valuations} \\
    $f^\varphi, f^\sgn, f^\nu, \text{etc.}$ & Apply $\varphi, \sgn, \nu$, etc.\ to each coefficient \ref{def:fphi} \\
    $\angular\colon K \to \kappa$ & Angular component map for a valued field \ref{def:angular map for a valued field} \\
    $\angular\colon H \rtimes \Gamma \to H$ & Angular component map for a tropical extension \ref{def:valuations} \\
    $\vangular\colon K \to \kappa \rtimes \Gamma$ & Refined valuation \ref{def:angular map for a valued field} \\
    $\vsign\colon K \to \S \rtimes \Gamma$ & Signed valuation \ref{def:signed valuation}\\
    $\PF$ & Polynomial function map \ref{def:polynomial-function} \\
\end{tabular}
\vspace{1 em}
\end{minipage}

\noindent
\begin{minipage}{\linewidth}
\begin{center} \textbf{Multiplicities} \end{center}
\begin{tabular}{ll}
    $\pmult^H$ & Perturbation multiplicity \ref{def:perturbation multiplicity} \\
    $\mult^\varphi$ & Relative multiplicity \ref{def:relative multiplicity} \\
    $\mult^H$  & Hyperfield multiplicity \ref{def:hyperfield multiplicity} \\
    $\bmult^H$ & Boundary multiplicity \ref{def:boundary multiplicity} \\
    $\gmult^H$ & $H$-enriched geometric multiplicity \ref{def:enriched geometric multiplicity} \\
    $N_{\bm h}^\varphi$ & Multiplicity for systems of equations \ref{sec:systems} \\
    $\sN_{\bm h}$ & Perturbation multiplicity for systems of equations \ref{def:perturbation multiplicity for systems}
\end{tabular}
\end{minipage}

\section{Fields and Hyperfields}

Hyperfields are algebraic objects which are well-suited to capture the arithmetic of signs (having forgotten the absolute value) or the arithmetic of absolute values (having forgotten the signs). One can think of a hyperfield as a field but where adding pairs of elements gives a non-empty set subject to the usual rules of commutativity, associativity, distributivity, etc.\ The axiom labeled ``reversible'' behaves as an ersatz subtraction.

\begin{definition}
    A \emph{hyperfield} is a tuple $H = (H, 0, 1, \cdot, \boxplus)$  where
    \begin{itemize}
        \item $0 \ne 1$,
        \item $H^* = (H \setminus \{0\}, 1, \cdot)$ is an Abelian group,
        \item $0$ is an absorbing element: $0 \cdot a = a \cdot 0$ for all $a \in H$.
    \end{itemize}
    Additionally, the \emph{hyperaddition} $\boxplus$ is a multivalued operation, that is a function $\boxplus \colon H \times H \to \{\text{nonempty subsets of } H\}$, such that for all $a, b \in H$:
    \begin{itemize}
        \item $a \boxplus b = b \boxplus a$ (commutative),
        \item $0 \boxplus a = \{a\}$ (identity),
        \item there is a unique element $-a$ such that $0 \in a \boxplus (-a)$ (inverses),
        \item $\bigcup\{ a \boxplus t : t \in b \boxplus c\} = \bigcup\{ t \boxplus c : t \in a \boxplus b\}$ (associative)
        \item $a \in b \boxplus c \iff -b \in (-a) \boxplus c$ (reversible)
    \end{itemize}

    Repeated addition is treated monadically, using the power set monad. This means that notationally we will identify elements of $H$ and singletons and repeated hyperaddition is flattened by unions---for example, $a \boxplus (b \boxplus c) = (a \boxplus b) \boxplus c$ means exactly what the associativity axiom says.
\end{definition}

In what follows, we will rarely need to work directly with the axioms above because we will use a common and more familiar subtype of hyperfields called  quotient hyperfields. All the hyperfields used in this paper are quotient hyperfields.

\begin{definition} \label{def:quotient-hyperfield}
    Let $F$ be a field and let $G$ be a subgroup of the group of units $F^*$. The \emph{quotient hyperfield} $F/G$ is the quotient set with the induced multiplication and the hyperaddition defined by
    \[
        aG \boxplus  bG = \{(c +d)G :  c\in aG \text{ and }d\in bG\} .
    \]
    If instead $F$ was a ring, then  $F/G$ is a \emph{quotient hyperring}.
\end{definition}

For simplicity of notation, we will often use the same symbols in $F$ to denote their equivalence classes in $F/G$. Furthermore, if $a \boxplus b$ is a singleton, we will omit the braces which indicate that the sum is a set.

\begin{example} \label{ex:krasner-hyperfield}
    If $F$ is any field with at least $3$ elements, then the hyperfield $\K=F/F^* = \{0, 1\}$ is called the \emph{Krasner hyperfield} after Marc Krasner. It has the following arithmetic:
    \[
        \begin{array}{c|cc}
        \cdot & 0 & 1 \\\hline
            0 & 0 & 0 \\
            1 & 0 & 1
        \end{array}
        \qquad
        \begin{array}{c|cc}
        \boxplus & 0 & 1 \\\hline
               0 & 0 & 1 \\
               1 & 1 & \K
        \end{array}
    \]
    The Krasner hyperfield is the hyperfield analogue of the Boolean semifield which has the same arithmetic except that $1 + 1 = 1$ instead of $\{0, 1\}$.
\end{example}

\begin{example} \label{ex:sign-hyperfield}
    The \emph{sign hyperfield} $\S = \R / \R_{> 0} = \{0, 1, -1\}$ is a quotient of the real numbers by the subgroup of positive real numbers. The arithmetic on $\S$ is given by the following tables.
    \[
        \begin{array}{c|ccc}
        \cdot & 0 & 1 & -1 \\\hline
            0 & 0 & 0 & 0 \\
            1 & 0 & 1 & -1 \\
            -1 & 0 & -1 & 1
        \end{array}
        \qquad
        \begin{array}{c|ccc}
        \boxplus & 0 & 1 & -1 \\\hline
               0 & 0 & 1 & -1 \\
               1 & 1 & 1 & \S \\
               -1 & -1 & \S & -1
        \end{array}
    \]
    This arithmetic encodes rules like ``positive times negative is negative'', ``negative plus negative is negative,'' and ``positive plus negative can be anything.''
\end{example}

\begin{example} \label{ex:tropical-hyperfield}
    If $(F, | \cdot |)$ is a field with an absolute value, then we can take its quotient by the group of elements with absolute value $1$ to create a hyperfield whose underlying set is the image $|F|$. The resulting hyperfield is called a \emph{triangle hyperfield} in the Archimedean case or an \emph{ultratriangle hyperfield} in the non-Archimedean case. Such hyperfields were first described by Viro who showed how they can be used to do computations in tropical geometry~\cite{V2}.

    The most common such hyperfield is where $| \cdot |$ is a non-Archimedean valuation whose image is $\R_{\ge 0}$. For our purposes, it will be more convenient to use the image of the associated valuation $\val(x) = -\log|x|$ (i.e.\ the set $\R \cup \{\infty\}$) as the base set instead. We call this the \emph{tropical hyperfield}, denoted by $\T$, where the arithmetic is given by $a \cdot_\T b = a +_\R b$ and
    \[
        a \boxplus b = \begin{cases}
            \min\{a, b\} & a \neq b, \\
            [a, \infty] & a = b.
        \end{cases} \qedhere
    \]
\end{example}

\subsection{Tropical Extensions}
\begin{example} \label{ex:tropical-extension}
    If $H$ is any hyperfield and $\Gamma$ is an ordered Abelian group, then we can extend $\Gamma$ by $H$ to get a version of the ultratriangle hyperfields of Example~\ref{ex:tropical-hyperfield} ``with coefficients in $H$.''

    Define the set
    \[ H \rtimes \Gamma = \{(h,\gamma) : h \in H^*, \gamma \in \Gamma\} \cup \{\infty\}. \]
    We will also use the notation $ht^\gamma = (h, \gamma)$ to better emphasize the relation between these extensions of hyperfields and extensions of a valued field $K$ to a valuation on $K(t)$ or $K((t))$ or similar (Remark~\ref{rem:hahn-series}).

    Multiplication is defined by $(h_1t^{\gamma_1})(h_2t^{\gamma_2}) = (h_1h_2)t^{\gamma_1 + \gamma_2}$ and the hypersum of $h_1t^{\gamma_1}$ and $h_2t^{\gamma_2}$ is defined as
    \begin{equation} \label{eq:trop-ext-addition}
        \begin{cases}
            h_1t^{\gamma_1} & \gamma_1 < \gamma_2, \\
            h_2t^{\gamma_2} & \gamma_2 < \gamma_1, \\
            (h_1\boxplus h_2)t^{\gamma_1} & \gamma_1 = \gamma_2 \text{ and } 0_H \notin h_1 \boxplus h_2, \\
            (h_1\boxplus h_2)t^{\gamma_1} \cup \{ht^\gamma : h \in H, \gamma > \gamma_1\} & \gamma_1 = \gamma_2 \text{ and } 0_H \in h_1 \boxplus h_2.
        \end{cases}
    \end{equation}
    We call this construction a \emph{tropical extension}.
\end{example}

\begin{remark}
    The hyperfield $\TR = \S \rtimes \R$ is called the \emph{tropical real hyperfield}. This hyperfield and other specific tropical extensions were first described in Viro's work \cite{V2}. The idea of extending ordered groups by a hyperfield appeared in the work of Bowler and Su~\cite{BS}. The tropical real hyperfield has also been used to describe real tropical geometry (e.g.\ \cite{JSY}).    
\end{remark}

\begin{remark}
    In terms of tropical extensions, we also have $\T = \K \rtimes \R$ and, in fact, every ultratriangle hyperfield described in Example~\ref{ex:tropical-hyperfield} is of the form $\K \rtimes \Gamma$ where $\Gamma$ is the image of the non-Archimedean valuation or absolute value.
\end{remark}

\begin{remark} \label{rem:hahn-series}
    If $H = F/G$ as in Definition~\ref{def:quotient-hyperfield}, then we can form the field of Hahn series
    \[
        F[[t^\Gamma]] = \left\{ \sum_{i \in I} a_it^i : a_i \in F \text{ and } I \text{ is a well-ordered subset of } \Gamma \right\}.
    \]
    There is a natural valuation $\nu$ on $F[[t^\Gamma]]$ given by $\nu(\sum_{i \in I} a_it^i) = \min\{i \in I : a_i \ne 0\}$.
    Now define
    \[
        G_0 = \left\{f = \sum_{i \in I} a_it^i \in F[[t^\Gamma]] : \nu(f) = 0_\Gamma \text{ and } a_0 \in G\right\}.
    \]
    The hyperfield $H \rtimes \Gamma$ is isomorphic to $F[[t^\Gamma]]/G_0$.
\end{remark}

\begin{remark}
    Bowler and Su \cite{BS} have a more general construction of a hyperfield from any extension
    \[
    1\to H^* \to G \to \Gamma \to 0
    \]
    of groups in which the conjugation operation of $G$ on $H^*$ extends to an action of $G$ on $H$ via automorphisms of hyperfields. In this context, $H\rtimes \Gamma$ is the hyperfield corresponding to the split extension of $\Gamma$ by $H^*$. Moreover, Bowler and Su show if $H\in \{\K,\S\}$, then all such extensions are split \cite[Theorem~4.17]{BS}. In a paper of the second author (TG), Bowler and Su's construction is described using the language of ordered blueprints \cite{G2}.
\end{remark}

\begin{remark}
    We can make the same definition if $\Gamma$ is an ordered semigroup instead of a group. If $\Gamma$ is not a group, then $H \rtimes \Gamma$ will be a hyperring instead of a hyperfield. This will be useful for us to talk about valuation hyperrings, which take the form $H \rtimes \Gamma_{\ge 0}$ with $\Gamma_{\ge 0} = \{\gamma \in \Gamma : \gamma \ge 0\}$.
\end{remark}

\subsection{Morphisms}

\begin{definition}
\label{def:morphism of hyperfields}
    A morphism between two hyperfields $H_1$ and $H_2$ is a map $\varphi \colon H_1 \to H_2$ such that for all $x, y \in H_1$:
    \begin{itemize}
        \item $\varphi(0) = 0$,
        \item $\varphi(1) = 1$,
        \item $\varphi(xy) = \varphi(x)\varphi(y)$,
        \item $\varphi(x \boxplus y) \subseteq \varphi(x) \boxplus \varphi(y)$.
    \end{itemize}
\end{definition}

\begin{lemma} \label{lem:morphism-over-sum}
    If $\varphi \colon H_1 \to H_2$ is a morphism of hyperfields and we have $A \in \bigboxplus_{i = 1}^n B_i C_i$ in $H_1$, then
    \[
        \varphi(A) \in \bigboxplus_{i = 1}^n \varphi(B_j)\varphi(C_j).
    \]
\end{lemma}
\begin{proof}
    By induction.
\end{proof}

\subsection{Valuations} \label{sec:valuations}

\begin{definition}
\label{def:valuations}
Let $H$ be a hyperfield. A valuation on $H$ is a morphism
\[
\nu\colon H\to \K \rtimes \Gamma
\]
of hyperfields for some totally ordered Abelian group $\Gamma$.
\end{definition}

\begin{example}
    \begin{enumerate}[label=(\alph*)]
    \item[]
    \item If $K$ is a field and $\nu\colon K\to \K \rtimes \Gamma$ is a map, then $\nu$ is a valuation in the sense of Definition~\ref{def:valuations} if and only if it is a valuation in the usual sense.
    \item For every hyperfield $H$ and every totally ordered Abelian group $\Gamma$, we obtain a valuation
    \[
    \nu\colon H \rtimes \Gamma\to \K \rtimes \Gamma,\;\; (h, \gamma) \mapsto \gamma .
    \]
    The map
    \[
    \angular \colon H \rtimes \Gamma\to H,\;\; (h, \gamma) \mapsto h
    \]
    is not a morphism of hyperfields in general. We call it the \emph{angular component} map
    \item For every hyperfield $H$ there is a unique morphism of hyperfields
    \[
    \nu_0\colon H\to \K .
    \]
    As $\K=\K \rtimes 0$, this is a valuation with value group $0$, the \emph{trivial valuation}. \qedhere
    \end{enumerate}
\end{example}

\begin{definition} \label{def:angular map for a valued field}
    Let $K$ be a valued field with valuation $\nu\colon K\to \K \rtimes \Gamma$ and residue field $\kappa$. Assume that the valuation $\nu\colon K\to \T$ splits, that is that there exists a morphism of Abelian groups $\psi\colon \Gamma\to K^*$ with $\nu(\psi(\gamma))=t^\gamma$. By abuse of notation, we denote $\psi(\gamma)=t^\gamma$. We define the \emph{angular component} (with respect to the given splitting) $\angular(a)$ of $a\in K^*$ by 
    \[
    \angular(a)=\overline{t^{-\nu(a)}a} \in \kappa ,
    \]
    where the bar indicates that we take the class in the residue field. We also set $\angular(0)=0$. We can then refine the valuation to a morphism of hyperfields
    \[
    \vangular\colon K\to \kappa \rtimes \R,\;\; a\mapsto 
    \begin{cases}
    \angular(a)t^{\nu(a)} & \text{, if }a\neq 0 \\
    0 &\text{, else}.
    \end{cases}
    \]
    By definition, we have $\angular(a)=\angular(\vangular(a))$ for every $a\in K$.
\end{definition}

Recall that a \emph{real closed field} is a field $K$ which is not algebraically closed and whose algebraic closure is $K(i)=K[x]/(x^2+1)$. Every real closed field is an ordered field, where the non-negative elements are precisely the squares. A \emph{valued real closed field} is a real closed field $K$ together with a valuation
\[
\nu\colon K\to \K \rtimes \Gamma
\]
such that $0<a<b$ implies $\nu(a)\geq \nu(b)$.
In this case, the residue field $\kappa$ is real closed again. If $\nu$ is surjective, then it splits \cite[Lemma 2.4]{TropicalSpectrahedra}. Since the angular component is multiplicative, we have 
\[
\sgn(a)=\sgn(\angular(a))
\]
for all $a\in K$. We define the \emph{signed valuation}\label{def:signed valuation} $\vsign$ as the composite
\[
K\xrightarrow{\vangular} \kappa \rtimes \Gamma\xrightarrow{\sgn \rtimes \Gamma} \S \rtimes \Gamma .
\]
By what we just observed, we have $\nu(\vsign(a))=\nu(a)$ and $\angular(\vsign(a))=\sign(a)$ for all $a\in K$.

\subsection{Real fields}
\begin{definition} \label{def:real-field}
    A hyperfield $R$, is called \emph{real} if it is equipped with a morphism $\sgn \colon R \to \S$. We call $\sgn$ a sign map on $R$.
\end{definition}

\begin{remark}
    Definition~\ref{def:real-field} mirrors Definition~\ref{def:valuations} and, in fact, both are special cases of ``valuations'' in the theory of ordered blueprints \cite[Chapter 6]{L}.
\end{remark}

\begin{remark}
    For any ordering $\le$ on a field $R$, there exists a unique morphism $\varphi \colon R \to \S$ such that $\varphi(x) = 1$ if $x > 0$ and $\varphi(x) = -1$ if $x < 0$. In fact, if $R$ is a ring, then morphisms $s \colon R \to \S$ correspond to pairs consisting of a prime ideal $\ker(s)$ and a total order on $R/\ker(s)$ \cite[Proposition 2.12]{ConnesConsani}. This concept can be extended to the language of schemes \cite{jun}.
\end{remark}

\begin{remark}
    Given a morphism from a field $K$ to $\TR$, we get both a total order on $K$ defined by the composition $K \to \TR \xrightarrow{\angular} \S$ and a valuation on $K$ defined by $K \to \TR \xrightarrow{\nu} \T$. The converse does not need to hold. For instance, $\Q$ has a natural total order and various $p$-adic valuations, but these $p$-adic valuations are not compatible with the total order. For a description of what makes a valuation compatible with a total order, we refer the reader to discussions in other papers \cite{G, AGT}.
\end{remark}

\section{Polynomials over hyperfields}
\begin{definition} \label{def:polynomials}
    If $H$ is a hyperfield and $\bm x = x_1,\dots,x_n$ are indeterminants, we define the set  of \emph{polynomials}
    \[ H[\bm x] = \left\{\sum a_{\bm m}{\bm x}^{\bm m} : \bm m \in \Z_{\geq 0}^n, \text{with finite support}\right\}, \] 
    where we use multi-index notation $\bm x^{\bm m} = x_1^{m_1}\cdots x_n^{m_n}$ and the \emph{support} of $f=\sum a_{\bm m} \bm x^{\bm m}$ is the set $\supp(f)=\{\bm m\in \Z_{\geq 0}^n  :  a_{\bm m}\neq 0\}$.
    Addition and multiplication (defined by convolution) give set-valued operations, meaning that $H[\bm x]$ is not, in general, a hyperfield.

    If $f, g, h \in H[\bm x]$ are such that $f \in g \cdot h$, we call this a \emph{factorization} of $f$. Concretely, if the coefficients of $f, g, h$ are $a_{\bm m}, b_{\bm m}, c_{\bm m}$, respectively, this means that for every $\bm m\in  \Z_{\geq 0}^n$ we have,
    \[ 
    a_{\bm m} \in \bigboxplus_{\bm n + \bm p = \bm m} b_{\bm n}c_{\bm p}. 
    \]

    If $f = \sum a_{\bm m} \bm x^{\bm m} \in H[\bm x]$ and $\bm z \in H^n$, then $f(\bm z)$ denotes the evaluation of $f$ at $\bm z$, which is the set $\bigboxplus a_{\bm m} \bm z^m$.
\end{definition}

\begin{remark}
    Because addition in hyperfields is set-valued, when we construct polynomials, both multiplication and addition are set-valued. We will make use of these operations, but we will not try to develop a broader theory of ring-like algebras with multivalued multiplication and addition for two reasons. First, $H[\bm x]$ is generally not ``free'' in the usual understanding of the adjective. Second, there is an existing theory due to Lorscheid of ``ordered blueprints'' which contains both hyperfields and free algebras, and which is a nicer and more natural setting to discuss polynomial algebras over hyperfields \cite{L}, \cite[Appendix]{BL}. See~\cite{G2} for a demonstration of how to rephrase hyperfield notation and multiplicities in terms of ordered blueprints.
\end{remark}

\begin{definition}
    In some examples, it will be convenient to use a grid notation for polynomials in two variables, where we put the coefficient of $x^iy^j$ at position $(i, j)$ and an empty space for a $0$ coefficient. For instance, the grid
    \[
        f = \begin{tikzpicture}[baseline=(current bounding box.center)]
        \matrix[matrix of math nodes]{
        + &   &     \\
        - &   &     \\
        + & - & +   \\
        };
    \end{tikzpicture}
    \]
    denotes the polynomial $+1 - x - y + x^2 + y^2\in \S[x,y]$.
\end{definition}

\begin{definition}
\label{def:fphi}
    Let $\varphi \colon H_1\to H_2$ be a morphism of hyperfields and let $f\in H_1[\bm x]$. We denote by $f^\varphi$ the polynomial in $H_2[\bm x]$ obtained by applying $\varphi$ to all coefficients of $f$.
\end{definition}

\begin{corollary} \label{cor:morphism-factorization}
    If $\varphi \colon H_1 \to H_2$ is a morphism of hyperfields, and $f \in g \cdot h$ in $H_1[\bm x]$, then $f^\varphi \in g^\varphi \cdot h^\varphi$.
\end{corollary}
\begin{proof}
    This follows directly from Lemma~\ref{lem:morphism-over-sum}.
\end{proof}

\begin{definition} \label{def:poly-morphism}
    Given two sets of polynomials $H_1[\bm x]$ and $H_2[\bm x]$, by a \emph{diagonal transformation}, $\Phi \colon H_1[\bm x] \to H_2[\bm x]$, we mean a function which is a composite of a map as in Definition~\ref{def:fphi} and a \emph{diagonal monomial substitution} of the form $\bm x \mapsto \bm a \bm x^{\bm k} = (a_1x_1^{k_1}, \dots, a_nx_n^{k_n})$ for some $\bm a \in H_2^n$ and $\bm k \in (\Z_{> 0})^n$.
\end{definition}

\begin{remark}
    More general monomial substitutions do not necessarily lead to element-to-element maps. For instance, substituting $y \mapsto x$ in $x + y$ yields $(1 \boxplus 1)x$. In the next lemma, we could also consider substitutions coming from injective semigroup homomorphisms $\N^n \to \N^n$ instead of just a diagonal ones but since the only substitutions we use have the form $\bm x \mapsto \bm {ax}$ or maybe relabelling some variables, it just makes for easier notation to only consider diagonal substitutions.
\end{remark}

\begin{lemma} \label{lem:mon-transf-factorization}
    If $f \in g \cdot h$ and $(x_i) \mapsto (a_ix_i^{k_i})$ is a diagonal monomial transformation, then $f(\bm a\bm x^{\bm k}) \in g(\bm a\bm x^{\bm k}) \cdot h(\bm a\bm x^{\bm k})$.
\end{lemma}
\begin{proof}
    Let $A_{\bm m}, B_{\bm n}, C_{\bm p}$ be the coefficients of $f, g, h$, respectively. So we have
    \[
        A_{\bm m} \in \bigboxplus_{\bm m = \bm n + \bm p} B_{\bm n} C_{\bm p}
    \]
    for all $\bm m\in \Z_{\geq 0}$.
    This implies that
    \[
        A_{\bm m} \bm a^{\bm m \bm k} \in \bigboxplus_{\bm m = \bm n + \bm p} B_{\bm n} C_{\bm p} \bm a^{\bm n \bm k + \bm p \bm k}
    \]
    which is the condition that $f(\bm a\bm x^{\bm k}) \in g(\bm a\bm x^{\bm k}) \cdot h(\bm a\bm x^{\bm k})$.
\end{proof}

Combining Corollary \ref{cor:morphism-factorization} and Lemma \ref{lem:mon-transf-factorization}, we obtain the following:

\begin{corollary}\label{cor:morphism-mult}
    If $\Phi \colon H_1[\bm x] \to H_2[\bm x]$ is a diagonal transformation and $f\in g\cdot h\in H_1[\bm x]$, then $\Phi(f)\in \Phi(g)\cdot \Phi(h)$.
\end{corollary}

\subsection{Newton Polygons}
A useful tool to understand the  combinatorics of polynomials over valued (hyper)fields is the \emph{Newton polytope}.
\begin{definition}
    Let $f = \sum a_{\bm m} \bm x^{\bm m} \in H[\bm x]$. We call the convex hull of $\supp(f) \subset \R^n$ the \emph{Newton polytope}, denoted $\Newt(f)$. We say that $f$ is \emph{dense} if $\supp(f)=\Newt(f)\cap\Z^m$. When $H$ has a valuation $v \colon H \to \T$, we furthermore have a subdivision of $\Newt(f)$, constructed as follows.
    
    Take the set of points
    \[ \mathcal S = \{(m, v(a_{\bm m})) \in \Z^m \times \R : \bm m \in \supp(f)\}. \]
    The \emph{lower convex hull} of $\mathcal{S}$ is the intersection of all ``lower-halfspaces'' containing $\mathcal{S}$. Here, a lower-halfspace is a halfspace cut out by a ``lower-inequality'': $\{p \in \R^{m + 1} : \langle u, p \rangle + c \ge 0\}$ for some $u \in \R^m \times \R_{\ge 0}$ and $c \in \R$.
    This lower convex hull is sometimes called the \emph{extended Newton polytope} of $f$.
    
    By projecting the faces of this extended Newton polytope into the first $m$ coordinates, we obtain a subdivision of $\Newt(f)$. For polynomials over valued hyperfields, $\Newt(f)$ refers to both the polytope and the subdivision, where appropriate.
\end{definition}

\begin{example}
    Consider the polynomial $1 + x + y + x^2 + xy + 1y^2 \in \T[x,y]$. The edges and vertices of the extended Newton polytope are drawn in Figure~\ref{fig:newton-subdivision} in greyscale and the associated subdivision is drawn beneath it in \textcolor{asypurple}{purple}.
    \begin{figure}[htbp]
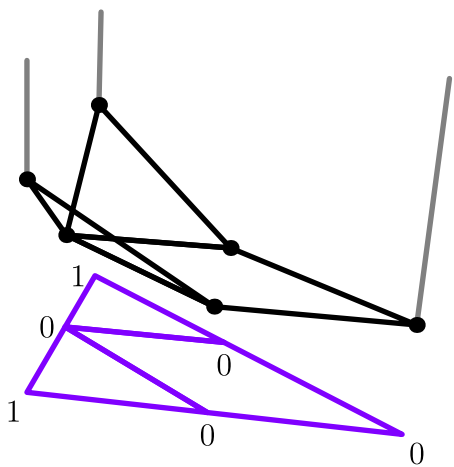

        \centering
        \begin{asy}
            import three;
            currentprojection = perspective(1, -4, 2.5);
            size(6cm, 6cm, keepAspect=false);
            defaultpen(linewidth(2));
            dotfactor = 3;
            settings.render = 0;
            
            draw((0,0,1) -- (0,0,2), grey);
            draw((0,2,1) -- (0,2,2), grey);
            draw((2,0,0) -- (2,0,2), grey);
            dot((0,0,1));
            dot((0,1,0));
            dot((0,2,1));
            dot((1,0,0));
            dot((1,1,0));
            dot((2,0,0)); 

            draw((0, 0, 1) -- (0, 1, 0) -- (1, 0, 0) -- cycle);
            draw((0, 1, 0) -- (1, 1, 0) -- (2, 0, 0) -- (1, 0, 0) -- cycle);
            draw((0, 1, 0) -- (1, 1, 0) -- (0, 2, 1) -- cycle);

            draw((0, 0, -1) -- (0, 1, -1) -- (1, 0, -1) -- cycle, purple);
            draw((0, 1, -1) -- (1, 1, -1) -- (2, 0, -1) -- (1, 0, -1) -- cycle, purple);
            draw((0, 1, -1) -- (1, 1, -1) -- (0, 2, -1) -- cycle, purple);

            label("$1$", (0, 0, -1), SW);
            label("$0$", (0, 1, -1), W);
            label("$1$", (0, 2, -1), W);
            label("$0$", (1, 0, -1), S);
            label("$0$", (1, 1, -1), S);
            label("$0$", (2, 0, -1), SE);
        \end{asy}
        \caption{Extended Newton polytope of the polynomial $f = 1 + x + y + x^2 + xy + 1y^2 \in \T[x,y]$ and \textcolor{asypurple}{associated subdivision} of $\Newt(f)$. Numbers indicate the valuation of the corresponding coefficient.}
        \label{fig:newton-subdivision}
    \end{figure}
\end{example}

\begin{definition}
    The Newton polytope of $1 + \sum_{i = 1}^n x_i$ is the standard $(n + 1)$-simplex, denoted $\Delta_{n + 1}$. The Newton polytope of $1 + \sum_{i = 1}^n x_i^d$ is denoted $d\Delta_{n + 1}$ and is the $d$-fold Minkowski sum of $\Delta_{n + 1}$. Concretely,
    \[
        d\Delta_{n + 1} = \left\{ \bm a \in \R_{\ge 0}^n \, : \, \sum a_i \le d \right\}.
    \]

    Given a polynomial $f \in H[\bm x]$, we say that $f$ has \emph{Newton-degree} $d$ if $\Newt(f) = d\Delta_{n + 1}$.
\end{definition}

\subsection{Polynomial Functions}

\begin{definition} \label{def:polynomial-function}
Every polynomial $f = \sum a_{\bm m} \bm x^{\bm m} \in \T[x_1,\dots,x_n]$ determines a tropical polynomial function $\PF_f$, given by
\[
\PF_f\colon \R^n\to \R ,\;\; \bm x\mapsto \min\{ a_{\bm m}+ \langle \bm m,x\rangle :  \bm m\in \Z_{\geq 0}^n\} .
\]
\end{definition}

Tropical polynomial functions are piecewise linear with integral slopes. We say that a monomial $a_{\bm m} \bm x^{\bm m}$ of $f$ is \emph{essential} if $\PF_f(\bm x)=a_{\bm m} + \langle \bm m, \bm x \rangle$ on some open subset of $\R^n$. In general, the polynomial $f$ is not determined by $\PF_f$, but all of its essential monomials are. More precisely, if $f^{\mathrm{ess}}$ denotes the sum of the essential monomials of $f$, then $\PF_f=\PF_{f^{\mathrm{ess}}}$. It follows that for two polynomials $f,g\in \T[\bm x]$ we have $\PF_f=\PF_g$ if and only if $f^{\mathrm{ess}}=g^{\mathrm{ess}}$. We say that $f$ is \emph{strictly convex} if $f=f^{\mathrm{ess}}$. Note that we always have $\Newt(f)=\Newt(f^{\mathrm{ess}})$.

\begin{remark} \label{rem:hyperfield-v-semifield}
    Polynomial functions use arithmetic from the tropical \emph{semifield} $\bar \R$ where $a \oplus b$ is the single element $\min\{a, b\}$.
    In Lorscheid's theory of ordered blueprints, there is a functor which relates the hyperfield $\T$ with the semifield $\bar \R$. Consider the order $\leqslant$ on $\T$, defined by $a \leqslant b + c$ if $a \in b \boxplus c$. If we add the relation $1 + 1 \leqslant 1$, we obtain $\bar\R$.
\end{remark}

\begin{lemma}
    \label{lem:additivity of PL functions}
    Let $f,g\in \T[\bm x]$ be polynomials and let $h\in f\cdot g$. Then we have
    \[
    \PF_h=\PF_f+\PF_g.
    \]
\end{lemma}

\begin{proof}
    Let $a_{\bm m}$, $b_{\bm m}$ and $c_ {\bm m}$ denote the coefficients of $f$, $g$, and $h$, respectively.
    Let $\bm w \in \R^n$ be generic; more precisely, we require that $\bm w$ is contained in the dense open subset of $\R^n$ where there exist unique $\bm m_1, \bm m_2\in \Z_{\geq 0}^n$ such that $\PF_f(\bm w)= a_{\bm m_1}+\langle \bm m_1, \bm w\rangle$ and $\PF_g(\bm w)= b_{\bm m_2}+\langle \bm m_2 , \bm w\rangle$. In particular, the minimum 
    \[
    \min\{a_{\bm m}+b_{\bm m'}+ \langle \bm m + \bm m', \bm w\rangle  :  \bm m, \bm m'\in \Z_{\geq 0}^n\}
    \]
    is attained exactly once, namely for $\bm m = \bm m_1$ and $\bm m' = \bm m_2$, and equal to $\PF_f(\bm w)+\PF_g(\bm w)$. Since for $k\in \Z_{\geq 0}$ we have $c_k\geq \min\{ a_{\bm m}+b_{\bm m'}  :  \bm m+\bm m'= k\}$, with equality if the minimum is attained exactly once, it follows that $c_{\bm m_1+\bm m_2}=a_{\bm m_1}+b_{\bm m_2}$ and that 
    \[
    \PF_h(\bm w)=c_{\bm m_1+\bm m_2}+\langle \bm m_1+\bm m_2 ,\bm w \rangle= \PF_f(\bm w)+\PF_g(\bm w) .
    \]
    By continuity of polynomial functions, this implies that $\PF_h=\PF_f+\PF_g$ on all of $\R^n$.
\end{proof}

\subsection{Initial forms}
Let $H$ be a hyperfield and $f\in (H \rtimes \R)[\bm x]$ and let $\bm w \in \R^n$. Consider the sub-hyperring $H \rtimes \R_{\ge 0} = \nu^{-1}(\R_{\geq 0} \cup \{\infty\})$ analogous to the valuation subring in a valued field. By definition of polynomial functions, we have 
\[
\widetilde f\coloneqq t^{-\PF_{f^\nu}(\bm w)}f(t^{w_1}x_1,\ldots t^{w_n}x_n) \in (H \rtimes \R_{\ge 0})[\bm x]
\]
and the minimum of the valuations of the coefficients of $\widetilde f$ is $0$. Denote
\[
r\colon H \rtimes \R_{\ge 0} \to H, \;\; (h, l) \mapsto
\begin{cases}
    0 & \text{if }  l>0, \\
    h & \text{else.}
\end{cases}
\]
One checks that $r$ is a morphism of hyperrings. The \emph{initial form} $\initial_{\bm w}(f)$ is defined as the image of $\widetilde f$ under $r$, that is
\[
\initial_{\bm w}(f)= (\widetilde f)^r .
\]

\begin{lemma}
\label{lem:taking initial forms respects products}
    Let $f,g\in (H \rtimes \R)[\bm x]$, let $\bm w \in \R^n$, and let $h\in f\cdot g$.  Then we have
    \[
        \initial_{\bm w}(\bm h)\in \initial_{\bm w}(f)\cdot \initial_{\bm w}(g) .
    \]
\end{lemma}

\begin{proof}
   By Lemma~\ref{lem:additivity of PL functions} we have $\PF_{h^\nu}(\bm w)=\PF_{f^\nu}(\bm w)+\PF_{g^\nu}(\bm w)$. It follows that 
   \begin{multline*}
        t^{-\PF_{h^\nu}(\bm w)}h(t^{w_1}x_1,\ldots, t^{w_n}x_n) \\
        \in \left(t^{-\PF_{f^\nu}(\bm w)}f(t^{w_1}x_1,\ldots, t^{w_n}x_n)\right)
            \left(t^{-\PF_{g^\nu}(\bm w)}g(t^{w_1}x_1,\ldots, t^{w_n}x_n)\right).
   \end{multline*}
   Applying the hyperring morphism $H \rtimes \R_{\ge 0} \to H$ to both sides of ``$\in$'' finishes the proof.
\end{proof}

We can then define the initial form of $f\in K[\bm x]$ at $\bm w \in \R^n$ by
\[
\initial_{\bm w}(f)=\initial_{\bm w}(f^\vangular) .
\]
This recovers the definition from the literature \cite[Chapter~2.4]{MS}.

\subsection{Tropical Hypersurfaces}

\begin{definition} \label{def:hyperfield-hypersurface}
    Let $f \in \T[\bm x]$ be a tropical polynomial. Its associated \emph{bend locus}, \emph{zero set}, \emph{variety} or \emph{hypersurface} is the set $V(f) = \{\bm b \in \R^n : f(\bm b) \ni \infty \}$.
\end{definition}

\begin{remark}
    Over a general hyperfield, one can also consider the zero set of a polynomial $f$ as $\{a \in H^n : f(a) \ni 0_H\}$. For our purposes, we defined the zero set as a subset of $\R^n = (\T^*)^n$ instead of $\T^n$ as that matches the more familiar definition of a tropical hypersurface~\cite{MS}.

    Such ``equations over hyperfields'' were first studied by Viro~\cite{V2}. For the tropical reals, Jell-Scheiderer-Yu reworded semialgebraic inequalities in terms of a polynomial containing a positive, non-negative, zero, etc. element of $\TR$~\cite{JSY}.
\end{remark}

For a polynomial $f\in \T[\bm x]$, the associated hypersurface, $V(f)$, carries a natural 
polyhedral structure. Namely, one defines $\bm w, \bm w'\in V(f)$ to be in the relative interior of the same polyhedron if and only if $\initial_{\bm w}(f)=\initial_{\bm w}(f')$. The facets of this polyhedral complex consist of precisely those points $\bm w$ for which $\initial_{\bm w}(f)$ is a binomial.

This is a \emph{weighted} polyhedral complex where, if $\initial_{\bm w}(f) = \bm x^{\bm a}+\bm x^{\bm b}$ is a binomial, the weight $V(f)[\sigma]$ of the facet $\sigma$ containing $\bm w$ is the integral length of $\bm a - \bm b$.
The polyhedral complex on $V(f)$, together with the weights on the facets, is called the tropical hypersurface of $f$. By abuse of notation, we also denote it by $V(f)$.

There is also a dual complex to $V(f)$, which is the polyhedral complex on the Newton polytope of $f$ whose non-empty polyhedra are the convex hull of the supports of polynomials of the form $\initial_{\bm w}(f)$ for $\bm w \in \R^n$. The components of $\R^n\setminus V(f)$ correspond to the vertices of the Newton subdivision, which in turn are precisely the exponents of the essential monomials of $f$. The facets of $V(f)$ correspond to the edges of the Newton subdivision.

While we described $V(f)$ in terms of $f$ for simplicity, it only depends on the polynomial function $\PF_f$. In fact, $V(f)$ determines $\PF_f$ up to a linear function. As polynomial functions can be added (tropical multiplication), this induces a sum of tropical hypersurfaces as well. The sum of two tropical hypersurfaces $V$ and $W$ can be described explicitly without reference to the defining polynomials (or polynomial functions). Namely, the underlying set of $V+W$ is $V\cup W$, and the weights are the sums of the weights of $V$ and $W$, where on  $V\setminus W$ we take the weight to be $0$, and similarly on $W\setminus V$.

\section{Factoring multivariate polynomials over hyperfields}
\subsection{The hyperfield multiplicity}

\begin{definition}
\label{def:hyperfield multiplicity}
Let $\cF, \cL \subseteq H[\bm x]$ be non-empty sets of polynomials over a hyperfield $H$ and assume that the degree is bounded on $\cF$ (i.e.\ there exists some $d>0$ such that all $f\in \cF$ have degree at most $d$). We let
    \[
    (\cF:\cL)=\{g\in H[\bm x] :  g\cdot l\cap \cF \neq \emptyset \text{ for some }l\in \cL\}.
    \]
Then we define the \emph{hyperfield multiplicity} $\mult^{H}_\cL(\cF)$ as follows: if $\cL$ contains a unit, we set $\mult^H_{\cL}(\cF)=\infty$. Otherwise, we define the multiplicity inductively as
    \[
    \mult^{H}_{\cL}(\cF)=
        \begin{cases}
            0   & \text{if } (\cF:\cL)=\emptyset, \\
            1 + \mult^H_\cL((\cF:\cL))         &\text{else.}
        \end{cases}
    \]
\end{definition}

If $\cL = \{l\}$ or $\cF = \{f\}$ are singletons, we will use the same notation without the braces, such as $(f : l)$ or $\mult^H_l(f)$.

\begin{remark}
    In most prior works, the multiplicity operator is defined for one polynomial and one linear factor. The exception to this is the work of Liu, which allows for a set of linear factors (but where $\cF$ is still a single polynomial) \cite{Liu}.
\end{remark}

\begin{example}
\label{ex:multiplicity over Krasner hyperfield}
If $H = \K$, and $l=1+\sum_{i=1}^n x_i\in \K[x_1,\ldots, x_n]$,
then $l\cdot \sum_{\vert \bm m \vert \leq d-1} \bm x^{\bm m}$ is the set of all polynomials over $\K$ of Newton-degree $d$. So if $f\in \K[\bm x]$ has Newton-degree $d$, then $\mult_l(f)=d$.
\end{example}

\begin{lemma}
\label{lem:hyperfield multiplicity of a set is a maximum}
Let $\cF, \cL \subseteq H[\bm x]$ be non-empty sets such that the degree is bounded on $\cF$. Then we have
\[
\mult^H_{\cL}(\cF) = \max\{\mult^H_{\cL}(f) :  f\in \cF\} .
\]
\end{lemma}

\begin{proof}
It follows directly from the definition of the multiplicity that if $\emptyset\neq \cF'\subseteq \cF$, then
\[
\mult^H_{\cL}(\cF')\leq \mult^H_{\cL}(\cF) .
\]
Therefore, we have  
\[
\mult^H_{\cL}(\cF)\geq \max\{\mult^H_{\cL}(f) :  f\in \cF\} .
\]
We show the reverse implication by induction on $\mult^H_{\cL}(\cF)$, the base case $\mult^H_{\cL}(\cF)=0$ being trivial. If $\mult^H_{\cL}(\cF)>0$, then we have
\[
\mult^H_{\cL}((\cF:\cL))=\max\{\mult^H_{\cL}(g) :  g\in (\cF:\cL)\}
\]
by the induction hypothesis. Let $g\in (\cF:\cL)$ be an element where this maximum is attained and let $f\in \cF$ and $l\in \cL$ such that $f\in g\cdot l$. Then we have
\begin{align*}
    \mult^H_{\cL}(f)
    =\mult^H_{\cL}((f:\cL))+1
    &\geq \mult^H_{\cL}(g)+1 \\
    &=\mult^H_{\cL}((\cF:\cL))+1 
    =\mult^H_{\cL}(\cF) . \qedhere
\end{align*}
\end{proof}

\begin{lemma}
\label{lem:multiplicities and morphisms}
Let $H_1$ and $H_2$ be hyperfields, let $\Phi \colon H_1[\bm x]\to H_2[\bm x]$ be a diagonal transformation. Let $\cL, \cF \subseteq H_1[\bm x]$ such that the degree is bounded on $\cF$. Suppose that $\Phi(\cF)$ does not contain the zero polynomial. Then we have
\[
\mult^{H_1}_\cL(\cF)\leq \mult^{H_2}_{\Phi(\cL)}(\Phi(\cF)).
\]
\end{lemma}

\begin{proof}
    Since the degree is bounded on $\cF$, it is also bounded on $\Phi(\cF)$. Also, if $\cL$ contains a unit, then so does $\Phi(\cL)$. Therefore, we may assume that neither $\cL$ nor $\Phi(\cL)$ contain a unit.

    The result now follows by induction from Corollary~\ref{cor:morphism-factorization} and Lemma~\ref{lem:mon-transf-factorization}. 
\end{proof}

\subsection{The boundary multiplicity}
For $i = 0,\dots, n$, let $\pi_i$ be the monomial transformation which substitutes $x_i \mapsto 0$ and $x_j \mapsto x_j$ for $j \ne i$. These monomial transformations are subject to Lemma~\ref{lem:multiplicities and morphisms}.

\begin{definition}
\label{def:boundary multiplicity}
Let $\cF, \cL \subseteq H[x_1,\ldots, x_n]$ be nonempty sets such that the degree on $\cF$ is bounded. Let $\widetilde{\cF}$ and $\widetilde{\cL}$ denote the polynomials in the variables $x_0,\ldots, x_n$ obtained by homogenizing the sets $\cF$ and $\cL$, respectively.
We define the \emph{boundary multiplicity} of $\cF$ at $\cL$ to be
\[
    \bmult^H_{\cL}(\cF) = \bmult^H_{\widetilde \cL}(\widetilde \cF) = \min \{\mult_{\pi_i(\widetilde \cL)}^{H}(\pi_i( \widetilde \cF))  :  0\leq i\leq n  \}
\]
\end{definition}

\begin{corollary}
\label{cor:mult leq bmult}
    Let $\cF, \cL \subset H[\bm x]$ be nonempty sets with bounded degree on $\cF$. We have
    \[
    \mult^H_\cL(\cF) \le \bmult^H_\cL(\cF).
    \]
\end{corollary}

\begin{proof}
Since multiplicities are not affected by homogenization, this follows directly from Lemma~\ref{lem:multiplicities and morphisms} applied to the morphisms $\pi_i$ for $0\leq i\leq n$.
\end{proof}

\begin{example}
\begin{enumerate}[label=(\alph*)]
    \item []
    \item If $f\in \K[\bm x]$ has Newton-degree $d$ and $l\in \K[\bm x]$ is the unique polynomial of Newton-degree $1$, then by Example~\ref{ex:multiplicity over Krasner hyperfield} we have
    \[
    \mult^\K_l(f)= \bmult^\K_l(f) = d.
    \]
    \item Let $f\in \S[x,y,z]$ be the degree-$3$ polynomial given by
\[
f=    \begin{tikzpicture}[baseline=(current bounding box.center)]
        \matrix[matrix of math nodes]{
        + &   &   &  \\
        - & + &   &  \\
        + & + & - &  \\
        + & + & - & +\\
        };
    \end{tikzpicture} 
\]
and let $l$ be the degree-$1$ polynomial given by
\[
l=
    \begin{tikzpicture}[baseline=(current bounding box.center)]
        \matrix[matrix of math nodes]{
        + &   \\
        + & + \\
        };
    \end{tikzpicture}.
\]
Then by the univariate Descartes' Rule of Signs~\cite[Example A.2]{G2}, \cite[Theorem C]{BL}, we have $\bmult^\S_l(f)=1$. We claim that $\mult^\S_l(f)=0$. Indeed, if $f\in g\cdot l$, then it follows from the conditions on the boundary that 
\[
  g= \begin{tikzpicture}[baseline=(current bounding box.center)]
        \matrix[matrix of math nodes]{
        + &   &     \\
        - & - &     \\
        + & - & +   \\
        };
    \end{tikzpicture}
\]
But for this choice of $g$, the $xy$-coefficient of any $h\in g\cdot l$ is necessarily negative, contradicting the fact that the $xy$-coefficient of $f$ is positive.
\end{enumerate}
\end{example}

\subsection{Multiplicities and initial forms}

\begin{example}
\label{ex:initial form}
    Let $f=\sum_{m\in \Z_{\geq 0}^n} a_{\bm m} x^{\bm m}\in (H \rtimes \R)[\bm x]$ be a polynomial in $n$-variables and let $\bm w \in \R^n$. Moreover, let $l=1+\sum_{i=1}^n t^{-w_i}x_i\in (H \rtimes \R)[\bm x]$. We have 
    \[
    \initial_{\bm w}(l)= 1+\sum_{i=1}^n x_i .
    \]
    In the univariate case (i.e.\ $n=1$), we have 
    \[
        \mult_l(f)=\mult_{\initial_{\bm w}(l)}(\initial_{\bm w}(f))
    \]
    by \cite[Theorem~A]{G2}. This cannot be true in higher dimensions by Lemma~\ref{lem:additivity of PL functions}. Concretely, it fails for the polynomial 
    \[
    f= 0+x+y+2x^2 + 1xy+ 2y^2 \in \T[x,y]
    \]
    and $w=0$. In this case, we have $\initial_0(f)=\initial_0(l)=1+x+y$ and hence $\mult_{\initial_{\bm w}(l)}(\initial_{\bm w}(f))=1$. On the other hand, $V(f)$ does not contain $V(l)$, as shown in Figure~\ref{fig:initial-form-counterexample}, and therefore $\mult_l(f)=0$ by Lemma~\ref{lem:additivity of PL functions}. We observe that
    \[
    \mult_l(f)\leq \mult_{\initial_{\bm w}(l)}(\initial_{\bm w}(f))
    \]
    in this example. 
\end{example}

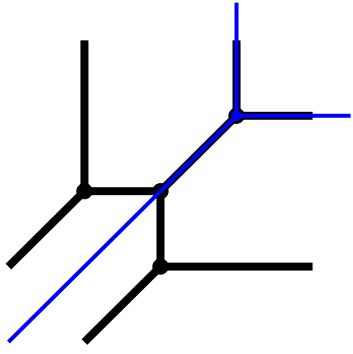
\begin{figure}[htbp]
    \centering
    \begin{tikzpicture}
        \foreach \x/\y in {0/0, -1/-1, -1/-2, -2/-1} {
            \filldraw (\x, \y) circle (0.1);
        }
        \draw[line width=3] (-3, -2) -- (-2, -1) -- (-1, -1) -- (0, 0) -- (0, 1);
        \draw[line width=3] (-2, -3) -- (-1, -2) -- (-1, -1);
        \draw[line width=3] (-2, -1) -- (-2, 1);
        \draw[line width=3] (-1, -2) -- (1, -2);
        \draw[line width=3] (0, 0) -- (1, 0);

        \filldraw[blue] (0, 0) circle (0.06);
        \draw[blue, line width=1.5] (0, 0) -- (0, 1.5);
        \draw[blue, line width=1.5] (1.5, 0) -- (0, 0) -- (-3, -3);
    \end{tikzpicture}
    \caption{Tropical curves defined by $0 + x + y + 2x^2 + 1xy + 2y^2$ and $\color{blue}{0 + x + y}$.}
    \label{fig:initial-form-counterexample}
\end{figure}

\begin{proposition}
   Let $H$ be a hyperfield, let $f\in (H \rtimes \R)[\bm x]$, and let $\bm w \in \R^n$. Moreover, let $\cL$ be a set of linear forms. Then we have 
    \[
        \mult_\cL(f)\leq \mult_{\initial_{\bm w}(\cL)}(\initial_{\bm w}(f)) , 
    \]
    where $\initial_{\bm w}(\cL)=\{\initial_{\bm w}(l) :  l\in \cL\}$.
\end{proposition}   

\begin{proof}
    This follows from Lemma~\ref{lem:taking initial forms respects products} and induction.
\end{proof}

In the case where the polynomial $f$ is defined over a field and factors as a product of linear forms, the initial forms contain considerably more information:

\begin{proposition}
\label{prop:initial ideal gives multiplicity for product of linear forms}
    Let $K$ be an algebraically closed valued field with residue field $\kappa$, let $f=\prod_{i=1}^d l_i\in K[\bm x]$ be a product of linear polynomials $l_i\in K[\bm x]$, and let $\bm w \in \R^n$. Moreover, let $l = 0 + \sum (-w_i) \cdot x_i\in \T[\bm x]$. Then we have
    \[
        \mult_{\nu^{-1}\{l\}}^K(f)=\mult_{\nu_0^{-1}\{\initial_{\bm w}(l)\}}^\kappa(\initial_{\bm w}(f))
    \]
\end{proposition}

\begin{proof}
After potentially scaling $f$ and the $l_i$, we may assume that the constant coefficient of each $l_i$, if it exists, is equal to $1$. Then the multiplicity $\mult_{\nu^{-1}\{l\}}^K(f)$ is equal to the number of $1\leq i\leq d$ such that $l_i^\nu=l$. Under the assumption on the constant coefficients, $l_i^\nu=l$ is equivalent to $\initial_{\bm w}(l_i)$ having support $\Delta_n$, which is equivalent to 
\[
\initial_{\bm w}(l_i)^{\nu_0} = 1 + \sum_{j=1}^n x_i=\initial_{\bm w}(l)\in \K[\bm x]
\]
Combining this with the fact that
\[
    \initial_{\bm w}(f)= \prod_{i=1}^d \initial_{\bm w}(l_i) 
\]
(Lemma~\ref{lem:taking initial forms respects products}), concludes the proof.
\end{proof}

\begin{lemma}
\label{lem:lifting real solutions of multiplicity 1}
    Let $K$ be a valued real closed field with residue field $\kappa$, and let $f=\prod_{i=1}^dl_i\in K[\bm x]$ be a product of linear polynomials $l_i\in \overline K[\bm x]$ over the algebraic closure $\overline K = K[\sqrt{-1}]$ of $K$. Furthermore, let $\bm w \in \R^n$ and assume that a degree-$1$ polynomial $\overline l\in \kappa[\bm x]$ divides $\initial_{\bm w}(f)$ with multiplicity $1$. Then there exists a degree-$1$ polynomial $l\in K[\bm x]$ dividing $f$ with $\initial_{\bm w}(l)=\overline l$. 
\end{lemma}

\begin{proof}
    We have $\initial_{\bm w}(f)=\prod_{i=1}^d\initial_{\bm w}(l_i)$ by Lemma~\ref{lem:taking initial forms respects products}.
    In particular, we may assume that after potentially renumbering and scaling by an appropriate element in $\overline K^*$, we have $\initial_{\bm w}(l_1)=\overline l$. It remains to show that $l_1\in K[\bm x]$. Let $\iota\colon \overline K\to \overline K$ denote complex conjugation. Then $f^\iota=f$, and therefore $l_1^\iota$ agrees with $l_j$ up to a constant factor for some $1\leq j\leq d$. It follows that $\initial_{\bm w}(l_j)$ and $\initial_{\bm w}(l_1)=\overline l$ differ by a constant. By the assumption that $\overline l$ divides $\initial_{\bm w}(f)$ with multiplicity $1$, we conclude that $j=1$. After potentially scaling by a constant, we may thus assume that $l_1^\iota=l_1$, that is that $l_1\in K[\bm x]$.
\end{proof}

\begin{proposition}
\label{prop:initial ideal gives multiplicity for product of linear forms, real case}
    Let $K$ be a valued real closed field with residue field $\kappa$. Suppose $f \in K[\bm x]$ factors as a product of linear forms $f = \prod_{i=1}^dl_i$ over the algebraic closure $\overline K=K[\sqrt{-1}]$ of $K$, and let $\bm w \in \R^n$. Moreover, let $l=1t^0+\sum s_it^{-w_i} x_i\in \TR[\bm x]$ for a choice of signs $s_i \in \S^*$. Assume that each factor of $\initial_{\bm w}(f)$ has multiplicity $1$. Then we have
    \[
        \mult_{\vsign^{-1}\{l\}}^K(f)=\mult_{\sgn^{-1}\{\initial_{\bm w}(l)\}}^\kappa(\initial_{\bm w}(f))
    \]
\end{proposition}

\begin{proof}
 We have
    \[
        \initial_{\bm w}(f)=\prod_{i=1}^d \initial_{\bm w}(l_i) .
    \]
    As a linear form  $g\in K[\bm x]$ is contained in $K_{>0}\cdot\vsign^{-1}\{l\}$ if and only if $\initial_w(g)\in \sign^{-1} \{\initial_{\bm w}(l)\}$, it follows that 
    \[
\mult_{\vsign^{-1}\{l\}}^K(f)\leq \mult_{\sgn^{-1}\{\initial_{\bm w}(l)\}}^\kappa(\initial_{\bm w}(f)) .
    \]
    The reverse inequality follows directly from Lemma~\ref{lem:lifting real solutions of multiplicity 1}.
\end{proof}

\subsection{The geometric multiplicity}

Suppose we have a hyperfield with valuation, say $H \rtimes \R$. Given a polynomial $f$ over $H \rtimes \R$, the valuation creates a tropical hypersurface $V(f)$. If $f$ has a linear factor, then we will have a linear component in this tropical hypersurface as well. Specifically, as observed in Example~\ref{ex:initial form}, it is a direct consequence of Lemma~\ref{lem:additivity of PL functions} that for any linear form $l$ and polynomial $f$ we have
\[
V(f)=\mult_l(f)\cdot V(l)+V(g) 
\]
for some polynomial $g$. This warrants the following definition.

\begin{definition} \label{def:geometric-multiplicity}
    Let $V$ be a tropical hypersurface and let $\cL\subseteq (H \rtimes \R)[\bm x]$ be a subset consisting of polynomials of degree $1$ that are not monomials. Then we define the \emph{geometric multiplicity}, $\gmult^\K_\cL(V)$, of $V$ with respect to $\cL$ to be
    \[
    \gmult^\K_{\cL}(V) = \max \sum_{i=1}^k a_i
    \]
    with the maximum taken over all $k$ and all $a_i\in \Z_{\geq 0}$ such that
    \[
    W + \sum_{i=1}^k a_i V(l_i^\nu) = V
    \]
    for some tropical hypersurface $W$ and some $l_i\in \cL$. For $f\in (H \rtimes \R)[\bm x]$ we abbreviate $\gmult^\K_{\cL}(V(f))=\gmult^\K_{\cL}(f)$.
\end{definition}

\begin{example} \label{ex:geometric multiplicity strictly larger}
    \phantom{link}
    \begin{enumerate} [label=(\alph*)]
        \item Let $f=0+x+y+1x^3+1x^2y +2y^3\in \T[x,y]$. As we see from the Newton subdivision shown in Figure~\ref{fig:gmsl}, the vanishing locus $V(f)$ is a union of $2$ tropical lines, one of which centered at the origin and one at $(-0.5,-1)$. So if $l=0+x+y$, then $\gmult^\K_l(f)=1$. On the other hand, we claim that $\mult_l(f)=0$. Indeed, assume that 
        \[
        f\in l\cdot (a+bx+cy+dx^2+exy+fy^2) .
        \]
        By looking at the coefficients of the constant term, $x^3$, and $y^3$, we see that we need to have $a=0$, $d=1$, and $f=2$. Because the coefficients of $f$ at $x^2$, $y^2$, and $xy^2$ are infinite, we also need to have $b=1$, $c=2$, and $e=2$. But then the $xy$-coefficient of $f$ is contained in $2+3+3=\{2\}$, a contradiction.
        \item 
        Let $f = +0 - x + y \in \TR[x,y]$ and $l = +0 + x + y$. Then $\gmult^\K_l(f)=1$, but $\mult_l^{\TR}(f)=0$. \qedhere
    \end{enumerate}
\end{example}

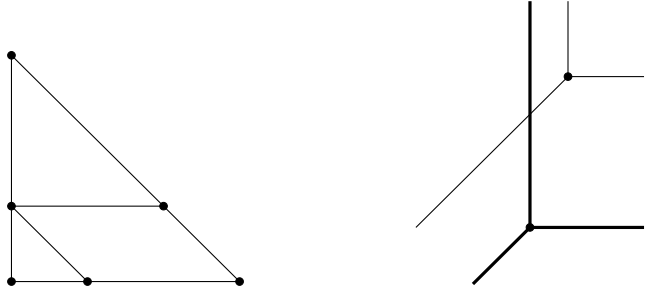
\begin{figure}[htb]
    \centering
    \begin{tikzpicture}
        \foreach \x/\y in {0/0, 1/0, 0/1, 3/0, 2/1, 0/3}
            \filldraw (\x, \y) circle (0.5mm);
        
        \draw (0,0) -- (3,0) -- (0,3) -- cycle;
        \draw (1,0) -- (0,1) -- (2,1);
    \end{tikzpicture}
    \hspace{2cm}
    \begin{tikzpicture}
        \filldraw (0, 0) circle (0.05);
        \draw (1, 0) -| (0, 1);
        \draw (0, 0) -- (-2, -2);
        \filldraw (-0.5, -2) circle (0.05);
        \draw[very thick] (1, -2) -| (-0.5, 1);
        \draw[very thick] (-0.5, -2) -- (-1.25, -2.75);
    \end{tikzpicture}
    \caption{Newton subdivision of $f = 0 + x + y + 1x^3 + 1x^2y + 2y^3$ and associated tropical curve $V(f)$.}
    \label{fig:gmsl}
\end{figure}

While both Example~\ref{ex:geometric multiplicity strictly larger} (a) and (b) show that the geometric multiplicity is, in general, larger than the multiplicity, the two examples are of a very different nature. Morally, in part (a) the reason for the discrepancy is that the vanishing locus of $f$ does not ``see'' all monomials of $f$ inside the Newton polytope, whereas in part (b) the reason is that the definition of geometric multiplicity of a polynomial over $H \rtimes \R$ only uses the valuation of the coefficients and does not use any information about $H$. To change this, we make the following definition.

\begin{definition}
    Let $H$ be a hyperfield. An \emph{$H$-enrichment} of a tropical hypersurface $A$ in $\R^n$, is an assignment of an element in $H^*$ to every connected component of $\R^n\setminus A$. Equivalently, it is a map $V\to H$, where $V$ is the set of vertices of the Newton subdivision corresponding to $A$. In particular, every $f\in (H \rtimes \R)[\bm x]$ induces an $H$-enriched tropical hypersurface $V(f)$.
    
    If $A$ and $B$ are two $H$-enriched tropical hypersurfaces, their sum $A+B$ is defined to have the sum of the underlying tropical hypersurfaces of $A$ and $B$ as the underlying tropical hypersurface, and the value of a connected component $C$ of $\R^n\setminus A+B$ is the product of the values of the connected components of $\R^n\setminus A$ and $\R^n\setminus B$ that contain $A$.
\end{definition}

\begin{remark}
    Enriched tropical hypersurfaces have also appeared in recent work of \cite{EnrichedTropicalIntersections} in the context of $\mathbf A^1$-geometry. In that setting, the components of the complement of a tropical hypersurface take values in the quotient hyperfield $k/(k^*)^2$ for some field $k$. 
\end{remark}

\begin{definition}
    An \emph{$H$-enriched tropical polynomial function} on $\R^n$ is a tropical polynomial function $f\colon \R^n\to \R$, together with an $H$-enrichment $s$ of $V(f)$.  The tropical product of two $H$-enriched tropical polynomial functions $(f,s)$ and $(g,s')$  is given by $(f+g,t)$, where $t$ is the enrichment of $V(f+g)$ obtained by adding the $H$-enriched hypersurfaces $(V(f),s)$ and $(V(g),t)$. Given a polynomial $f\in (H \rtimes \R)[\bm x]$ in $n$ variables, the polynomial function $\PF_{f^\nu}$ is naturally $H$-enriched: on each component $C$ of $\R^n\setminus V(f)$, a unique monomial, say $at^w x^{\bm m}$, of $f^\nu$ is minimized, and we assign to $C$ the value $a\in H^*$. We denote by $\PF_f$ the $H$-enriched polynomial function obtained this way.
\end{definition}  

\begin{lemma}
    \label{lem:additivity of enriched PL functions}
    Let $f,g\in (H \rtimes \R)[\bm x]$ and let $h\in f\cdot g$. Then 
    \[
        \PF_{h}=\PF_f\odot\PF_g
    \]
    as $H$-enriched tropical polynomial functions. In particular, we have
    \[
    V(h)=V(f)+V(g) .
    \]
\end{lemma}

\begin{proof}
    By Lemma~\ref{lem:additivity of PL functions}, we only need to show that the $H$-enrichments on both sides coincide. Let $C$ be a component of $\R^n\setminus V(h^\nu)$ and suppose the unique monomials of $f$ and $g$ that are minimized on $C$ are $M_1=at^{w_1}\bm x^{\bm m_1}$ and $M_2=bt^{w_2}{\bm x}^{\bm m_2}$, respectively. Let $f'$ and $g'$ be the polynomials obtained from $f$ and $g$ by omitting $M_1$ and $M_2$, respectively, then 
    \[
    h\in M_1M_2 +M_1g' + M_2f'+f'g' .
    \]
    By construction, we have for any point $\bm w \in C$ that $\PF_{f}(\bm w)=\PF_{M_1}(\bm w)<\PF_{f'}(\bm w)$ and $\PF_g(\bm w)=\PF_{M_2}(\bm w)<\PF_{g'}(\bm w)$. Therefore, 
    \[
    \PF_{M_1M_2}(\bm w)< \PF_{M_1g' + M_2f'+f'g'}(\bm w) ,
    \]
    from which we conclude that $M_1M_2$ is the unique monomial of $h$ minimized at $\bm w$ (and hence on $C$) and that the enrichment of $h$ on $C$ is given by $a\cdot b$, which is precisely the product of the enrichments of $f$ and $g$ there.

    The statement about hypersurfaces follows immediately from the statements about polynomial functions and the fact that $V(h^\nu)=V(f^\nu)+V(g^\nu)$.
\end{proof}

We can now define an enriched version of the geometric multiplicity, completely analogous to the geometric multiplicity.

\begin{definition}
\label{def:enriched geometric multiplicity}
Let $V$ be an $H$-enriched tropical hypersurface and let $\cL\subseteq (H \rtimes \R)[\bm x]$ be a subset consisting of linear forms. Then we define the \emph{$H$-enriched geometric multiplicity} $\gmult^H_\cL(V)$ of $V$ with respect to $\cL$ to be 
    \[
    \gmult^H_{\cL}(V) = \max \sum_{i=1}^k a_i
    \]
    with the maximum taken over all $k$ and all $a_i\in \Z_{\geq 0}$ such that
    \[
    W + \sum_{i=1}^k a_i V(l_i) = V
    \]
    for some $H$-enriched tropical hypersurface $W$ and some $l_i\in \cL$. For $f\in (H\rtimes \R)[\bm x]$ we abbreviate $\gmult^H_\cL(V(f))=\gmult^H_\cL(f)$.
\end{definition}

\begin{remark}
\label{rem:gmult over Krasner}
    Since $\K^*$ only consists of one element, tropical hypersurfaces and $\K$-enriched tropical hypersurfaces are equivalent. In particular, for $H=\K$ the definition of $\gmult^\K$ of Definition~\ref{def:enriched geometric multiplicity} agrees with the definition of $\gmult^\K$ from Definition~\ref{def:geometric-multiplicity}.
\end{remark}

\begin{lemma}
\label{lem:multipliciy bounded by geometric multiplicity}
Let $f\in (H \rtimes \Gamma)[\bm x]$ and let $\mc L\subseteq (H\rtimes\Gamma)[\bm x]$ be a set of polynomials of degree $1$ that are not monomials. Then we have
\[
\mult^{H\rtimes \Gamma}_{\mc L}(f)\leq \gmult^H_{\mc L}(f).    
\]
\end{lemma}

\begin{proof}
    The assertion is a direct consequence of Lemma~\ref{lem:additivity of enriched PL functions}.
\end{proof}

\begin{example}
    \begin{enumerate}[label=(\alph*)]
        \item []
        \item 
        As noted in Remark \ref{rem:gmult over Krasner}, geometric multiplicity and enriched geometric multiplicity coincide over $\K$. In particular, Example~\ref{ex:geometric multiplicity strictly larger}~(a)
        can be seen as an example where the enriched geometric multiplicity is strictly smaller than the multiplicity. Morally speaking, any discrepancy between the geometric multiplicity and (hyperfield) multiplicity in that example is entirely due to the valuations, replacing geometric multiplicity with enriched geometric multiplicity will not reduce the discrepancy.
        \item Let $f=0-x+y\in \TR$ and $l=0+x+y$, as in Example~\ref{ex:geometric multiplicity strictly larger}. Then $\gmult^\S_l(f)=\gmult^\K_l(f)=0$.
    \end{enumerate} 
\end{example}

\begin{lemma}
\label{lem:comparison between enriched and unenriched geometric multiplicity}
Let $V\subseteq \R^n$ be an $H$-enriched tropical hypersurface and let $l\in (H \rtimes \R)[\bm x]$ be a linear form. If $\gmult^\K_l(V)>1$, then $\gmult^H_l(V)\geq 1$. In particular, we either have $\gmult^H_l(V)=\gmult^\K_l(V)$ or $\gmult^H_l(V)=\gmult^\K_l(V)-1$.
\end{lemma}

\begin{proof}
    Let $W$ be the unique tropical hypersurface  with $W+V(l^\nu)=V$ as tropical hypersurfaces. Because $\gmult^\K_l(V)>1$, we have $V(l^\nu)\subseteq W$, and hence $\R^n\setminus V=\R^n\setminus W$. Denote by $s$ and $t$ the enrichments of $V$ and $V(l)$, respectively. Let $C$ be a component of $\R^n\setminus W$ and let $C'$ be the unique component  of $\R^n\setminus V(l^\nu)$ containing $C$. Then we can enrich $W$ by assigning to $C$ the element $s(C)\cdot t(C')^{-1}\in H^*$. By construction, we then have $W+V(l^\nu)=V$ as enriched tropical hypersurfaces. This shows that $\gmult^H_l(V) \geq 1$. The remainder of the assertion follows by induction.
\end{proof}

\begin{definition}
\label{def:strictly convex}
    We call a polynomial $f\in (H \rtimes \R)[\bm x]$ \emph{strictly convex} if $f^\nu\in \T[\bm x]$ is strictly convex.
\end{definition}

\begin{proposition}
\label{prop:multiplicity in strictly convex case}
    Let $\Gamma$ be a subgroup of $\R$, let $H$ be a hyperfield, let $f \in (H \rtimes \Gamma)[\bm x]$ be a dense strictly convex polynomial, and let $l\in  (H \rtimes \Gamma)[\bm x]$ be a degree-$1$ polynomial that is not a monomial and such that $\gmult^H_l(f)>0$. Then there exists a unique polynomial $g \in (H \rtimes \Gamma)[\bm x]$ with and $f\in g\cdot l$ and in fact $g$ is dense, strictly convex, and we have $\{f\}=g\cdot l$. 
\end{proposition}

\begin{proof}
    Let $W$ be an enriched tropical hyperplane such that $W+V(l)=V(f)$ and let $g\in (H \rtimes \Gamma)[\bm x^{\pm 1}]$ with $V(g)=W$. Then $V(\PF_g\odot \PF_l)=W+V(l)=V(\PF_f)$ and therefore $\PF_g\odot\PF_l$ and $\PF_f$ differ by a linear function. After multiplying $g$ by a suitable monomial, we may thus assume that $\PF_g\odot \PF_l=\PF_f$. 
    For every $h\in g\cdot l$, we have $\PF_h=\PF_f$ by Lemma~\ref{lem:additivity of enriched PL functions}. But since $f$ is dense and strictly convex this is only possible if $f=h$. 
    We conclude that $g\cdot l=\{f\}$. 

    Now let $g'\in (H \rtimes \Gamma)[\bm x^{\pm 1}]$ with $f\in g'\cdot l$. We will first show that $g'$ is strictly convex. Let $P$ be a maximal polytope in the Newton subdivision of $g$. It corresponds to some vertex $\bm p$ of $V(g)$. Let $Q$ be the polytope in the Newton subdivision of $l$, corresponding to the stratum of $V(l)$ containing $\bm p$. Then the polytope in the Newton subdivision of $f$ corresponding to $\bm p$ is given by the Minkowski sum $P+Q$. Let $\bm v$ be a vertex of $Q$ and let $\bm w$ be a lattice point contained in $P$. Then $\bm w + \bm v$ is a lattice point of $P+Q$. Because $f$ is dense and strictly convex, this implies that $\bm w+\bm v$ is a vertex of $P+Q$ and hence a vertex of $P+\bm v$. Therefore, $\bm w$ is a vertex of $P$. We conclude that every lattice point in the Newton polytope of $g'$ is a vertex of the Newton subdivision of $g'$, which implies that $g'$ is dense and strictly convex. We can now show that $g'=g$. Because 
    \[
    \PF_g\odot \PF_l =\PF_f=\PF_{g'}\odot \PF_l ,
    \]
    we have $\PF_g=\PF_{g'}$. But by what we just showed, both $g$ and $g'$ are strictly convex and hence uniquely determined by their enriched polynomial functions. We conclude that $g=g'$.

    Finally, note that $l$ has order $0$ with respect to each of the variables $x_i$. Therefore, the order of $g$ coincides with the order of $f$ with respect to each of the variables $x_i$. It follows that $g$ is a polynomial, that is $g\in (H \rtimes \Gamma)[\bm x]$.
\end{proof}

\begin{corollary}
    \label{cor:enriched multiplicity equals multiplicity in strictly convex case}
    Let $\Gamma$ be a subgroup of $\R$, let $H$ be a hyperfield, and let $f \in (H \rtimes \Gamma)[\bm x]$ be a dense strictly convex polynomial. Moreover, let $\mc L\subseteq (H\rtimes\Gamma)[\bm x]$ be a set of degree-$1$ polynomials not containing a monomial. Then we have
    \[
    \gmult^H_{\mc L}(f)=\mult^{H \rtimes \Gamma}_{\mc L}(f) .
    \]
\end{corollary}

\begin{proof}
    By Lemma~\ref{lem:multipliciy bounded by geometric multiplicity}, we need to show that
    \[
    \gmult^H_{\mc L}(f)\leq \mult^{H \rtimes \Gamma}_{\mc L}(f) .
    \]
    We do induction on $n=\gmult^H_{\mc L}(f)$, the base case $n=0$ being trivial. For $n>0$, there exists an $H$-enriched tropical hypersurface $W$ and a polynomial $l\in \mc L$  with $\gmult^H_{\mc L}(W)=n-1$ and $W+V(l)=V(f)$. In particular $\gmult^H_l(f)>0$. By Proposition~\ref{prop:multiplicity in strictly convex case}, there exists a dense strictly convex polynomial $g\in (H\rtimes \Gamma)[\bm x]$ with $f\in g\cdot l$. In particular, we have $V(f)=V(g)+V(l)$ by Lemma~\ref{lem:additivity of enriched PL functions} and hence $V(g)=W$. Using the induction hypothesis, we conclude that
    \[
    \gmult^H_{\mc L}(f)=1+\gmult^H_{\mc L}(g)\leq 1+\mult^{H\rtimes\Gamma}_{\mc L} (g)\leq \mult^{H\rtimes\Gamma}_{\mc L}(f) . \qedhere
    \]
\end{proof}

\subsection{Relative hyperfield multiplicity}

\begin{definition}
\label{def:relative multiplicity}
    Let $\varphi\colon H_1\to H_2$ be a morphism of hyperfields and let $\emptyset\neq \mc F, \mc L\subseteq H_2[\bm x]$ such that the degree is bounded on $F$. The \emph{relative multiplicity of $\mc F$ at $\mc L$ with respect to $\varphi$}, denoted by $\mult^\varphi_{\mc L}(\mc F)$, is given by
    \[
        \mult^\varphi_{\mc L}(\mc F)= \mult_{\varphi^{-1}{\mc L}}^{H_1}(\varphi^{-1}\mc F) .
    \]
\end{definition}

\begin{proposition}
\label{prop:comparison relative multiplicity and hyperfield multiplicity}
    Let $\varphi\colon H_1\to H_2$ be a morphism of hyperfields and let $\emptyset\neq \mc F,\mc L\subseteq H_2[\bm x]$ such that the degree is bounded on $\mc F$. Then we have
    \[
    \mult^\varphi_{\mc L}(\mc F)\leq \mult^{H_2}_{\mc L}(\mc F).
    \]
\end{proposition}

\begin{proof}
    This is follows immediately from Lemma~\ref{lem:multiplicities and morphisms} applied to the morphism $H_1[\bm x]\to H_2[\bm x]$ induced by $\varphi$.
\end{proof}

\begin{example}
\label{ex:relative multiplicty can be way off}
\begin{enumerate}[label=(\alph*)]
    \item []
    \item Let $K$ be a field and let $\nu_0\colon K \to \K$ be the trivial valuation. Let $d\in \Z_{>0}$ be coprime to the characteristic of $K$, and let $f=1+x^d+y^d$ and $l=1+x+y$ be elements in $\K[x,y]$. We have already seen in Example~\ref{ex:multiplicity over Krasner hyperfield} that $\mult^\K_l(f)=d$. To compute the relative multiplicity with respect to $\nu_0$, let $g=a+bx^d+cy^d\in K[x,y]$ be any polynomial with $g^{\nu_0}=f$. Since $a+cy^d$ has only simple roots, Eisenstein's criterion, applied with respect to any prime factor of $a+cy^d$, shows that $g$ is irreducible. We conclude that 
\[
\mult_{\nu_0^{-1}\{l\}}^K(g)=\begin{cases}
    1 &\text{if }d=1, \\
    0 &\text{else},
\end{cases}
\]
and therefore
\[
\mult_l^{\nu_0}(f)=\begin{cases}
    1 &\text{if }d=1, \\
    0 &\text{else}.
\end{cases}
\]
\item We keep the setting of part (a), but instead take $f=\sum_{\vert m\vert \leq d}x^{\bm m}$. If $K$ is infinite, then for $d$ generic linear forms $l_1,\ldots, l_d\in \nu_0^{-1}\{l\}$ we have  $\left(\prod_{i=1}^d l_i\right)^{\nu_0}=f$, and hence $\mult^{\nu_0}_l(f)=\mult^\K_l(f)=d$. If the field $K$ is finite, things are more complicated. For example, if $K=\mathbf F_2$ and $d=2$, then  $\mult^{\nu_0}_l(f)=0$. \qedhere
\end{enumerate}
\end{example}

\begin{example}
\label{ex:hyperfield multiplicity larger than relative multiplicity}
    For the morphism $\sgn\colon \R\to \S$, the hyperfield multiplicity can be strictly larger than the relative hyperfield multiplicity, even for dense polynomials. 
Consider the polynomial
\[
f=    \begin{tikzpicture}[baseline=(current bounding box.center)]
        \matrix[matrix of math nodes]{
        + &   &   &  \\
        - & + &   &  \\
        + & - & - &  \\
        + & - & + & +\\
        };
    \end{tikzpicture} \in
    \begin{tikzpicture}[baseline=(current bounding box.center)]
        \matrix[matrix of math nodes]{
        + &   \\
        + & + \\
        };
    \end{tikzpicture} \cdot
    \begin{tikzpicture}[baseline=(current bounding box.center)]
        \matrix[matrix of math nodes]{
        + &   &     \\
        - & - &     \\
        + & - & +   \\
        };
    \end{tikzpicture}.
\]
The given factorization of $f$ is the unique way to factor out $l = 1+x+y$, so we see that $\mult^\S_l(f)=\bmult^\S_l(f)=1$. However, there exists no degree-$2$ polynomial $g\in \R[x,y]$ such that $(1+x+y)g$ has the given sign pattern. Assume on the contrary that such $g$ existed. We may assume that $g(0,0)=$, and write $g(x,y)=1-ax-by+cx^2-dxy+ey^2$, where $a,b,c,d,e$ are positive reals. Then we have
\begin{multline*}
(1+x+y)g(x,y)=1+(1-a)x+ (1-b)y+(c-a)x^2+ (-a-b-d)xy+ (e-b)y^2 + \\
+ cx^3+ (c-d)x^2y+ (-d+e)xy^2+ey^3 .
\end{multline*}
This product having the signs of $f$ is equivalent to
\begin{align*}
    1& <a &
    1 &> b \\
    c& > a &
    e &< b\\
    c& <d  &
    e& > d,
\end{align*}
from which we obtain a chain
\[
1 < a < c < d < e < b < 1.
\] 
A contradiction!
\end{example}

\begin{proposition}
\label{prop:relative multiplicity in strictly convex case}
    Let $K$ be a field, $H$ a hyperfield, $\Gamma\subseteq \R$ a totally ordered group, and let $\varphi\colon K\to H \rtimes \Gamma$ be a surjective morphism of hyperfields. Moreover, let $f\in  (H \rtimes \Gamma)[\bm x]$ be a dense strictly convex polynomial, and let $\mc L\subseteq (H \rtimes \Gamma)[\bm x]$ be a set of polynomials of  Newton-degree $1$. Then we have
    \[
    \mult^\varphi_{\mc L}(f)=\mult^{H \rtimes \Gamma}_{\mc L}(f) .
    \]
\end{proposition}

\begin{proof}
By  Proposition~\ref{prop:comparison relative multiplicity and hyperfield multiplicity}, we have $\mult^\varphi_{\mc L}(f)\leq\mult^{H \rtimes \Gamma}_{\mc L}(f)$. We show the reverse inequality by induction on $m=\mult^{H \rtimes \Gamma}_{\mc L}(f)$. The base case $m=0$ is trivial, so we may assume that $m>0$, in which case we have $m=1+\mult^{H \rtimes \Gamma}_{\mc L}((f:\mc L))$. By Lemma~\ref{lem:hyperfield multiplicity of a set is a maximum}, there exists $g\in (f:\mc L)$ with $\mult^{H \rtimes \Gamma}_{\mc L}((f:\mc L))=\mult^{H \rtimes \Gamma}_{\mc L}(g)$, and by definition of $(f:\mc L)$ we have $f\in g\cdot l$ for some $l\in \mc L$. By Proposition~\ref{prop:multiplicity in strictly convex case}, the polynomial $g$ is dense, strictly convex, and $g\cdot l=\{f\}$, so by the induction hypothesis we have
\[
\mult^{H \rtimes \Gamma}_{\mc L}(g)= \mult^{\varphi}_{\mc L}(g)= \mult^K_{\varphi^{-1}\mc L}(\varphi^{-1}\{g\}) .
\]
Again by Lemma~\ref{lem:hyperfield multiplicity of a set is a maximum}, there exists $\widetilde g\in \varphi^{-1}\{g\}$ with 
\[
\mult^K_{\varphi^{-1}\mc L}(\varphi^{-1}\{g\}) =\mult^K_{\varphi^{-1}\mc L}(\widetilde g) .
\]
Let $\widetilde l\in \varphi^{-1}\{l\}$. Then we have 
\[
(\widetilde g\cdot \widetilde l)^\varphi \in g\cdot l=\{f\}  ,
\]
that is $(\widetilde g\cdot \widetilde l)^\varphi=f$. It follows that 
\[
\mult^\varphi_{\mc L}(f)=\mult^K_{\varphi^{-1}\mc L}(\varphi^{-1}\{f\})
\geq \mult^K_{\varphi^{-1}\mc L}(\widetilde g\cdot \widetilde l)
\geq 1+\mult^K_{\varphi^{-1}\mc L}(\widetilde g)
=m .
\]
\end{proof}

\subsection{Perturbation multiplicity}
One technique for analyzing the roots of a polynomial in $\C[\bm x]$ is to perturb the coefficients within the field of Puiseux series $\C[[t^\Q]]$ and consider a homotopy as $t \to 0$. By analogy, if we want to compute a multiplicity over a hyperfield $H$, we can consider the same multiplicity in $H \rtimes \R$ after a small perturbation. We will only consider strictly convex pertubations; in the case where the polynomial $f\in H[\bm x]$ we start with is dense, this allows us to bound the multiplicity of $f$ from below by $H$-enriched geometric multiplicities, which are much easier to compute than hyperfield multiplicities.

For this multiplicity, we work over $\S$. The sign hyperfield is special in that the inclusion $\S \to \S \rtimes \R = \TR$ splits canonically. That is, the angular component map $\angular \colon \TR  \to \S$ is a morphism of hyperfields. 

\begin{remark}
    A tropical extension consists of an exact sequence of groups $1 \to H^* \to E^* \to \Gamma \to 1$ meaning $\operatorname{im}(H^* \to E^*) = \operatorname{eq}(1, E^* \to \Gamma)$. The corresponding sequence of hyperrings $0 \to H \to E \to \K \rtimes \Gamma \to 0$ is not necessarily exact because $\operatorname{eq}(1, E* \to \Gamma)$ is only the multiplicative kernel. So despite having a section $\Gamma \to H \rtimes \Gamma, \gamma \mapsto t^\gamma$, we should not expect that the angular component map $\angular \colon H \rtimes \Gamma \to H$ is a morphism.
\end{remark}

\begin{definition}
\label{def:perturbation multiplicity}
    Let $f \in \S[\bm x]$ and let $l\in \S[\bm x]$ be a linear form. Let $\mc F$ denote the subset of $\angular^{-1}\{f\}$ consisting of strictly convex polynomials in $\TR[\bm x]$. We define the \emph{perturbation multiplicity} of $l$ in $f$, denoted $\pmult_l^\S(f)$ by
    \[
    \pmult_l^\S(f)= \mult^{\TR}_{\angular^{-1}\{l\}}(\mc F) .
    \]
\end{definition}

\begin{corollary}
\label{cor:comparison of multiplicity with perturbation multiplicity}
    Let $f\in \S[\bm x]$ and let $l\in \S[\bm x]$ be a linear form. Then we have
    \[
    \pmult_l^\S(f)\leq \mult^\S_l(f) .
    \]
    If $f$ is dense, $\mc F \subset \TR[\bm x]$ is the set of all strictly convex polynomials in $\angular^{-1}(f)$, and $l$ is not a monomial, then 
    \[
    \pmult_l^\S(f)=\gmult^{\S}_{\angular^{-1}\{l\}}(\mc F)
    \]
\end{corollary}

\begin{proof}
    The inequality is a direct consequence of Lemma~\ref{lem:multiplicities and morphisms}, the equality a direct consequence of Corollary~\ref{cor:enriched multiplicity equals multiplicity in strictly convex case}.
\end{proof}

\begin{remark}
\label{rem:permutation multiplicity and subdivisions}
    Given a dense polynomial $f \in \S[\bm x]$ and a linear form $l\in \S[\bm x]$, the equality $\pmult_l^\S(f)=\gmult^{\S}_{\angular^{-1}\{l\}}(\mc F)$ from Corollary \ref{cor:comparison of multiplicity with perturbation multiplicity} reduces the computation of $\pmult_l^\S(f)$ to a finite problem, that is only finitely many multiplicities $\gmult_{\angular^{-1}\{l\}}^\S(\widetilde f)$ for $\widetilde f\in \mc F$ need to be computed. Indeed, the condition that $V(\widetilde f)=W+V(\widetilde l)$ for some $\S$-enriched tropical hypersurface $W$ and some $\widetilde l\in \angular^{-1}\{l\}$ does not depend on the exact position of the vertices of the $\S$-enriched tropical hypersurface $V(\widetilde f)$, but only its combinatorial type. Expressed dually, $\gmult^\S_{\angular^{-1}\{l\}}(\widetilde f)$ only depends on $l$, $f$, and the Newton subdivision of $\widetilde f$, for which there are only finitely many choices.

    Now assume we are in two variables and we are given a strictly convex $\widetilde f$ in $\angular^{-1}\{f\}$. If $V(\widetilde f)=W+V(\widetilde l)$ as above, then the Newton subdivision of $f$ is a mixed subdivision of the Newton subdivisions of $W$ and $V(\widetilde l)$. Because $\widetilde f$ is dense and strictly convex, every lattice point of $\Newt(\widetilde f)$ appears as a vertex of the Newton subdivision of $\widetilde f$.  This can only happen if $W$ and $V(\widetilde l)$  meet transversally with intersection multipliciy $1$. Therefore, every cell in the mixed subdivision of $W$ and $V(\widetilde l)$ either is a translate of a cell in the Newton subdivision of $W$ or $V(\widetilde l)$, or a parallelogram of volume $1$. Since $V(\widetilde f)=W+V(\widetilde l)$ needs to hold on the level of $\S$-enriched tropical hypersurfaces, the signs of $f$ and $l$ give additional constraints on which mixed subdivisions can appear for $f$. Namely, each translate of a cell of the Newton subdivision of $W$ and $V(\widetilde l)$ has to have the same signs as in $W$ or $V(\widetilde l)$ or exactly opposite signs, and each parallelogram has to be of the following form, up to translation and the action of $\mathrm{GL}_2(\Z)$:
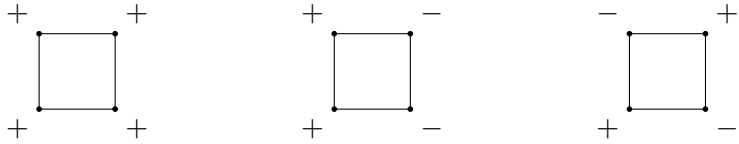
    \begin{center}
        \begin{tikzpicture}[baseline=0]
            \draw (0, 0) -- (0, 1) -- (1, 1) -- (1, 0) -- cycle;
            \filldraw (0,0) circle (0.03) node[below left] {$+$};
            \filldraw (1,0) circle (0.03) node[below right] {$+$};
            \filldraw (0,1) circle (0.03) node[above left] {$+$};
            \filldraw (1,1) circle (0.03) node[above right] {$+$};
        \end{tikzpicture}
        \hspace{15mm}
        \begin{tikzpicture}[baseline=0]
            \draw (0, 0) -- (0, 1) -- (1, 1) -- (1, 0) -- cycle;
            \filldraw (0,0) circle (0.03) node[below left] {$+$};
            \filldraw (1,0) circle (0.03) node[below right] {$-$};
            \filldraw (0,1) circle (0.03) node[above left] {$+$};
            \filldraw (1,1) circle (0.03) node[above right] {$-$};
        \end{tikzpicture}
        \hspace{15mm}
        \begin{tikzpicture}[baseline=0]
                \draw (0, 0) -- (0, 1) -- (1, 1) -- (1, 0) -- cycle;
                \filldraw (0,0) circle (0.03) node[below left] {$+$};
                \filldraw (1,0) circle (0.03) node[below right] {$-$};
                \filldraw (0,1) circle (0.03) node[above left] {$-$};
                \filldraw (1,1) circle (0.03) node[above right] {$+$};
        \end{tikzpicture}
    \end{center}
\end{remark}
    
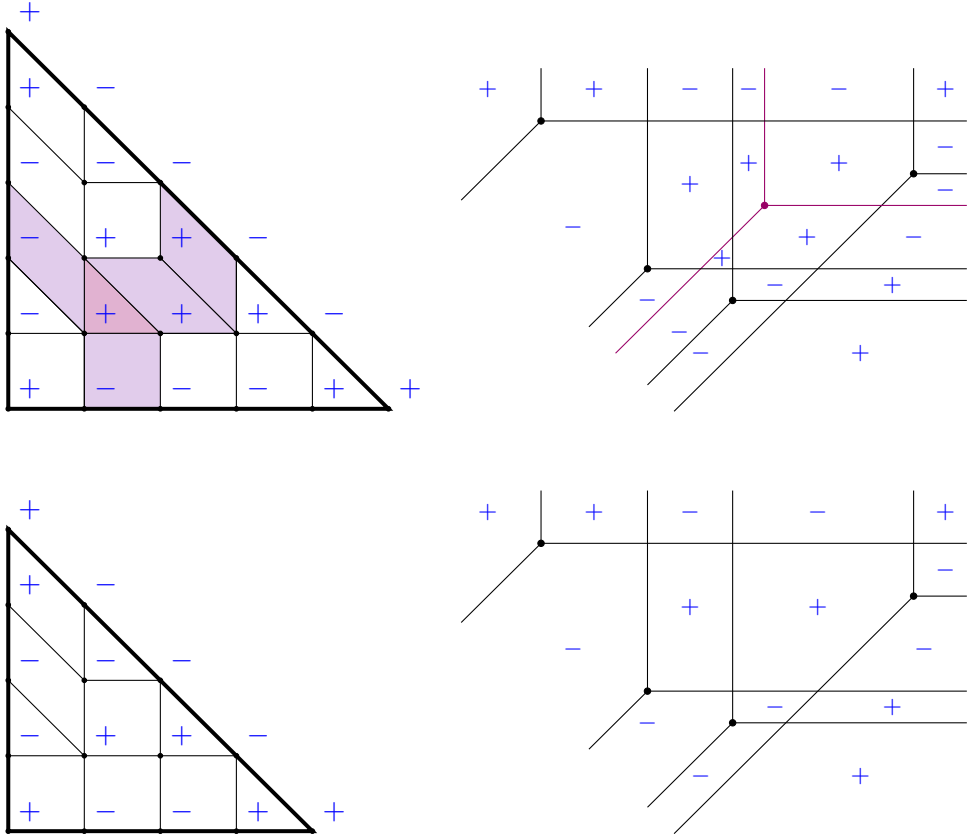
\begin{figure}[tbp]
    \centering
        \begin{tikzpicture}[scale=1]
        \filldraw [fill=red!60!blue!30] (1,1) -- +(1,0) -- +(0,1) -- cycle;
        \filldraw [fill=red!40!blue!20] (1,0) -| +(1,1) -| cycle;
        \filldraw [fill=red!40!blue!20] (0,2) -- +(1,-1) -- +(1,0) -- +(0,1) -- cycle;
        \filldraw [fill=red!40!blue!20] (2,2) -- +(1,-1) -- +(1,0) -- +(0,1) -- cycle;
        \filldraw [fill=red!40!blue!20] (2,1) -- +(-1,1) -- +(0,1) -- +(1,0) -- cycle;
    
        \draw (0,1) -| (1,0);
        \draw (0,2) -- +(1,-1);
        \draw (0,4) -- +(1,-1);
        \draw (4,0) -- +(0,1);
        \draw (1,4) |- +(1,-1);
        \draw (3,2) |- +(1,-1);
        \draw (2,1) -| +(1,-1);
        \draw (1,2) -- +(0,1);
    
        \foreach \x in {0,...,5} {
            \foreach \y [parse=true] in {0,...,5 - \x}{
                \filldraw (\x, \y) circle (0.03);
            }
        }

        \foreach \x/\y in {0/0, 0/4, 0/5, 1/1, 1/2, 2/1, 2/2, 3/1, 4/0, 5/0} {
            \draw (\x, \y) node[above right, color=blue] {$+$};
        }

        \foreach \x/\y in {0/1, 0/2, 0/3, 1/0, 1/3, 1/4, 2/0, 2/3, 3/0, 3/2, 4/1} {
            \draw (\x, \y) node[above right, color=blue] {$-$};
        }
    
        \draw [line width=1.5] (0,0) -- (0, 5) -- (5, 0) -- cycle;
    \end{tikzpicture}
    \hfill
 \begin{tikzpicture}[font=\footnotesize, scale=1.4, rotate=180]
    \begin{scope}[color=blue]
        \node at (.2,.2) {$+$};
        \node at (1.2,.2) {$-$};
        \node at (2.05,.2) {$-$};
        \node at (2.6,.2) {$-$};
        \node at (3.5,.2) {$+$};
        \node at (4.5,.2) {$+$};

        \node at (.2,.75) {$-$};
        \node at (1.2,.9){$+$};
        \node at (2.05,.9){$+$};
        \node at (2.6, 1.1){$+$};
        \node at (3.7, 1.5){$-$};
        
        \node at (.2,1.15) {$-$};
        \node at (1.5, 1.6) {$+$};
        \node at (2.3,1.8) {$+$};
        \node at (3,2.2){$-$};
          
        \node at (.5, 1.6){$-$};
        \node at (1.8,2.05){$-$};
        \node at (2.7,2.5){$-$};
        
        \node at (.7, 2.05){$+$};
        \node at (2.5,2.7){$-$};
        
        \node at (1, 2.7){$+$};
    \end{scope}
        
        \maxtropline{0.5}{1}{6}
        \maxtroplinecolor{1.9}{1.3}{6}{red!60!blue}
        \maxtropline{2.2}{2.2}{6}
        \maxtropline{3}{1.9}{6}
        \maxtropline{4}{0.5}{6}
    \end{tikzpicture}    

    \vspace{1cm}
    
    \begin{tikzpicture}[scale=1]   
        \draw (0,1) -| (1,0);
        \draw (1,1) -| (2,0);
        \draw (2,1) -| (3,0);
        \draw (0,2) -- (1,1) -- (1,3);
        \draw (0,3) -- (1,2) -- (2,2) -- (2,1);
    
        \foreach \x in {0,...,4} {
            \foreach \y [parse=true] in {0,...,4 - \x}{
                \filldraw (\x, \y) circle (0.03);
            }
        }

        \foreach \x/\y in {0/0, 0/3, 0/4, 1/1, 2/1, 3/0, 4/0} {
            \draw (\x, \y) node[above right, color=blue] {$+$};
        }

        \foreach \x/\y in {0/1, 0/2, 1/0, 1/2, 1/3, 2/0, 2/2, 3/1} {
            \draw (\x, \y) node[above right, color=blue] {$-$};
        }
    
        \draw [line width=1.5] (0,0) -- (0, 4) -- (4, 0) -- cycle;
    \end{tikzpicture}
    \hfill
 \begin{tikzpicture}[font=\footnotesize, scale=1.4, rotate=180]
    \begin{scope}[color=blue]
        \node at (.2,.2) {$+$};
        \node at (1.4,.2) {$-$};
        \node at (2.6,.2) {$-$};
        \node at (3.5,.2) {$+$};
        \node at (4.5,.2) {$+$};

        \node at (.2,.75) {$-$};
        \node at (1.4,1.1){$+$};
        \node at (2.6, 1.1){$+$};
        \node at (3.7, 1.5){$-$};
        
        \node at (.4,1.5) {$-$};
        \node at (3,2.2){$-$}; 
        \node at (1.8,2.05){$-$};
        
        \node at (.7, 2.05){$+$};
        \node at (2.5,2.7){$-$};
        
        \node at (1, 2.7){$+$};
    \end{scope}
        
        \maxtropline{0.5}{1}{6}
        \maxtropline{2.2}{2.2}{6}
        \maxtropline{3}{1.9}{6}
        \maxtropline{4}{0.5}{6}
    \end{tikzpicture}

   \caption{Sign compatible subdivision, quotient with induced subdivision, and associated tropical hypersurfaces.}
    \label{fig:signed-subdivision}
\end{figure}

\begin{example}
    With the notation as in Remark~\ref{rem:permutation multiplicity and subdivisions}, let $\widetilde f\in \TR[x,y]$ be a polynomial of Newton-degree $5$ with $f^\angular$ and its Newton subdivision as in  Figure~\ref{fig:signed-subdivision} on the top left. Then the Newton subdivision can be realized as a mixed subdivision of subdivisions of the $4$-simplex and the Newton polytope of $l=1+x+y$ (the $1$-simplex) by declaring the triangle in dark purple in the figure as the unique unmixed cell coming from the $1$-simplex, and declaring the light purple cells as the mixed cells. The dark purple unmixed cell has the same sign pattern as the Newton polytope of $l$ and the mixed cells all have the allowed sign patterns outlined in Remark~\ref{rem:permutation multiplicity and subdivisions}. We can conclude that $V(\widetilde f)=W+V(\widetilde l)$ for some $\S$-enriched tropical hypersurface $W$ and some $l\in \angular^{-1}\{l\}$. Moreover, the procedure determines the subivision and signs of the Newton polytope of $W$: simply remove the cells in purple and push together the remaining cells. The result is depicted on the lower left of Figure~\ref{fig:signed-subdivision}. Note that this procedure can be repeated with the all-negative triangle and suitably chosen mixed cells, giving a total geometric multiplicity of $\gmult^\S_{\angular^{-1}\{l\}}(\widetilde f)=2$.

    Finally, the right of Figure~\ref{fig:signed-subdivision} shows the dual tropical picture. The given Newton subdivision of $\widetilde f$ makes $V(\widetilde f^\nu)$ a union of tropical lines. The tropical line $L$ in purple on the top right corresponds to the purple cells and what we phrased in terms of subdivisions above is that there exists an $\S$-enrichment  $\widetilde L$ of $L$ and an $\S$-enriched tropical hypersurface $W$ such that $V(\widetilde f)=W+\widetilde L$ and $\widetilde L=V(\widetilde l)$ for some $\widetilde l \in \angular^{-1}\{l\}$. The $\S$-enriched tropical hypersurface $W$ is depicted on the bottom right.
\end{example}

\begin{example}
    The perturbation multiplicity can also be defined over hyperfields $H$ for which the angular component $\angular\colon H\rtimes \R\to H$ is not a morphism. However, in these settings the inequality $\pmult_l^H(f)\leq \mult^H_l(f)$ will fail to hold in general. Consider the polynomial
    \[
    f(x,y)=0+1 x+y+1 x^2+1 xy+ y^2 \in \T[x,y]
    \]
    and let $l=0+x+y\in \T[x,y]$. Then $\mult^\T_l(f)=\gmult^\T_l(f)=0$.
    Now extend from $\T$ to $\T \rtimes \R$ (using reverse lexicographic order). We have
    \begin{align*}
        &[(0,0)+(0,0)x+(0,0)y] \cdot [(0,0) + (1, -1)x + (0,1) y] \\
        &\qquad = (0, 0)
                + (1,-1)x
                + (0,0) y
                + (1,-1)x^2
                + (1,-1)xy
                + (0,1)y^2
    \end{align*}
    This is a strictly convex polynomial whose (coefficient-wise) angular component is $f$, so $\pmult^\T_l(f)\geq 1$.
\end{example}

\begin{proposition}
\label{prop:smult leq rel-mult}
    Let $f\in \S[\bm x]$ be dense and let $l\in \S[\bm x]$ be of Newton-degree $1$. Moreover, let $K$ be a valued real closed field with value group $\R$. Then we have
    \[
    \pmult_l^\S(f)\leq \mult^{\sgn}_l(f).
    \]
\end{proposition}

\begin{proof}
     By Lemma~\ref{lem:hyperfield multiplicity of a set is a maximum}, there exists a polynomial $g\in \angular^{-1}\{f\}\subseteq \TR[\bm x]$, which is strictly convex and where $\pmult^\S_l(f)= \mult^{\TR}_{\angular^{-1}\{l\}}(g)$. Because $f$ is dense, $g$ is dense as well. By the definition of the relative multiplicity and Proposition~\ref{prop:relative multiplicity in strictly convex case}, we have
    \[
    \mult^K_{\sgn^{-1}\{l\}}(\vsign^{-1}\{g\})=\mult^\vsign_{\angular^{-1}\{l\}}(g)= \mult^{\TR}_{\angular^{-1}\{l\}}(g) .
    \]
    As $\vsign^{-1}\{g\}\subseteq \sgn^{-1}\{f\}$, we conclude that
    \begin{align*}
    \mult^\sgn_l(f)&=\mult^K_{\sgn^{-1}\{l\}}(\sgn^{-1}\{f\}) \\
    &\geq\mult^K_{\sgn^{-1}\{l\}}(\vsign^{-1}\{g\})=\mult^{\TR}_{\angular^{-1}\{l\}}(g) . \qedhere
    \end{align*}
\end{proof}

\begin{example}
\label{ex:relative mult can be strictly larger than perturbation mult}
The perturbation multiplicity can be strictly smaller than the relative multiplicity with respect to $\sgn$, even for dense polynomials. To see this, consider the polynomial
\[
f=    \begin{tikzpicture}[baseline=(current bounding box.center)]
        \matrix[matrix of math nodes]{
        - &   &   &  \\
        - & + &   &  \\
        + & - & - &  \\
        + & + & + & -\\
        };
    \end{tikzpicture} \in
    \begin{tikzpicture}[baseline=(current bounding box.center)]
        \matrix[matrix of math nodes]{
        + &   \\
        + & + \\
        };
    \end{tikzpicture} \cdot
    \begin{tikzpicture}[baseline=(current bounding box.center)]
        \matrix[matrix of math nodes]{
        - &   &     \\
        - & + &     \\
        + & + & -   \\
        }; 
    \end{tikzpicture}
\]
and let $l=1+x+y$.
The given factorization of $f$ is the unique way to factor out $l$, so we see that $\mult_l^\S(f)=\bmult_l^\S(f)=1$. We also have
\[
f=\big((1+x+y)(1+.5x-.3y)(1-.33x+.01y)\big)^\sgn ,
\]
so that $\mult_l^\sgn(f)=1$ as well.
However, there is no signed mixed subdivision containing a positive or negative triangle, so $\pmult_l^\S(f)=0$. 
\end{example}

\subsection{Multiplicities over \texorpdfstring{$\S$}{S} in degree 2}

Since multiplicities in degree $1$ are trivial, we now study in detail the first interesting case of polynomials of Newton-degree $2$. We work entirely over the hyperfield $\S$.

\begin{proposition}
\label{prop:bmult equals mult for quadratic polynomials}
    Let $H$ be a hyperfield, let $f\in H[\bm x]$ be a   polynomial of Newton-degree $2$ in $n\geq 2$ variables and let $l\in \S[\bm x]$ be of Newton-degree $1$. Then we have
    \[
    \bmult^\S_l(f)= \mult^\S_l(f) .
    \]
\end{proposition}

\begin{proof}
    To simplify notation, we homogenize both $l$ and $f$, introducing a new variable $x_0$. After scaling the variables appropriately, we may further assume that $l=\sum_{i=0}^n x_i$. Let $A$ be the support of $f$ and write $f=\sum_{\bm a\in A} c_{\bm a} \bm x^{\bm a}$. Let $h= \sum_{i=0}^n c_{2\bm e_i} x_i$, where $\bm e_0,\ldots, \bm e_n$ denotes the standard basis of $\Z^{n+1}$.
    Whenever $f\in l\cdot g$, the square terms $c_{2\bm e_i}x_i^2$, of $f$ uniquely determine $g$. More precisely, $f\in l\cdot g$ implies that $g=h$.
    
    For $0\leq i\leq n$ let $\pi_i\colon H[x_0,\ldots, x_n]\to H[x_0, \ldots, \hat x_i,\ldots, x_n]$ be the morphism sending $x_i$ to $0$ and $x_j$ to $x_j$ for $j\neq i$. For each $0\leq i\leq n$, the polynomial $\pi_i(f)$ also has Newton-degree $2$. Therefore, the same reasoning as for $f$ applies to $\pi_i(f)$ and $\pi_i(f)\in \pi_i(l)\cdot g$ implies $g=\pi_i(h)$. Because all monomials of $f$ only involve  two variables and $n\geq 2$, we have $f\in l\cdot h$ if and only if $\pi_i(f)\in \pi_i(l)\cdot \pi_i(h)$ for all $0\leq i\leq n$. By what we have observed, this implies that $\mult^H_l(f)\geq 1$ is equivalent to $\bmult^H_l(f)\geq 1$. Moreover, we have $\mult^H_l(f)=2$ if and only if $\mult^H_l(f)\geq 1$ and $h$ and $l$ differ by a factor in $H^*$. On the other hand, $h$ and $l$ differ by a factor in $H^*$ if and only if $\pi_i(h)$ and $\pi_i(l)$ differ by a factor in $H^*$ for all $0\leq i\leq n$, so that $\mult^H_l(f)=2$ is equivalent to $\bmult^H_l(f)=2$. 
\end{proof}

\begin{theorem}
\label{thm:multiplicity for dense quadratic poly in 2 variables}
Let $f\in \S[x,y]$ be a dense polynomial of Newton-degree $2$ and let $l\in \S[x,y]$ be of Newton-degree $1$. Then we have
\[
\pmult_l^\S(f)=\mult^\sgn_l(f)=\mult_l^\S(f)= \bmult_l^\S(f) .
\]
\end{theorem}

\begin{proof}
In light of the inequalities from Proposition~\ref{prop:smult leq rel-mult}, Proposition~\ref{prop:comparison relative multiplicity and hyperfield multiplicity}, and Corollary~\ref{cor:mult leq bmult}, it suffices to show that
\[
\pmult_l^\S(f)=\bmult_l^\S(f) .
\]
    There are $64$ dense polynomials in $\S[x,y]$ of Newton-degree $2$, but using symmetry we can group these into $4$ cases.
    First, consider the corners of the Newton polytope. By multiplying everything by $-1$, we may assume that either $2$ or $3$ of the corners are $+$. Additionally, if we view these sign arrangements as a homogeneous polynomial $f(x,y,z) \in \S[x,y,z]$ then we can make use of the symmetries $x \leftrightarrow y$, $x \leftrightarrow z$ and $y \leftrightarrow z$ to permute the corners arbitrarily.
    This splits the $64$ polynomials into two categories:
    \[
        \begin{matrix}
            + &   &   \\
            * & * &   \\
            + & * & +
        \end{matrix}
       \qquad \text{ and } \qquad\quad
        \begin{matrix}
            + &   &   \\
            * & * &   \\
            + & * & -
        \end{matrix}.
    \]

    Secondly, we have the symmetries $x \leftrightarrow -x$, $y \leftrightarrow -y$ and $z \leftrightarrow -z$ which affect the middle signs as indicated in Figure~\ref{fig:middle-symmetries}.
    \begin{figure}[htbp]
        \centering
        \begin{tikzpicture}
            \matrix[matrix of math nodes]{
                + &   &   \\
                |(Y)| * & |(X)| * &   \\
                + & |(Z)| * & - \\
            };
            \draw[<-] (Y.west) -- ++(-0.5,0) node[left] {row$\times (-1)$};
            \draw[<-] (X.north) -- ++(0,0.5) node[above] {column$\times (-1)$};
            \draw[<-] (Z.south east) -- ++(0.4,-0.4) node[below right] {diagonal$\times (-1)$};
        \end{tikzpicture}
        \caption{Transformations $x \leftrightarrow -x, y \leftrightarrow -y, z \leftrightarrow - z$.}
        \label{fig:middle-symmetries}
    \end{figure}
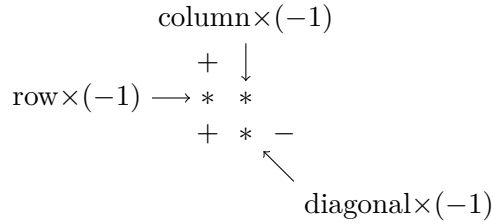
    Using these symmetries, we can assume that at least $2$ of the middle signs are $+$, and that leaves us with just $4$ cases which we number as in Figure~\ref{fig:4cases}.

    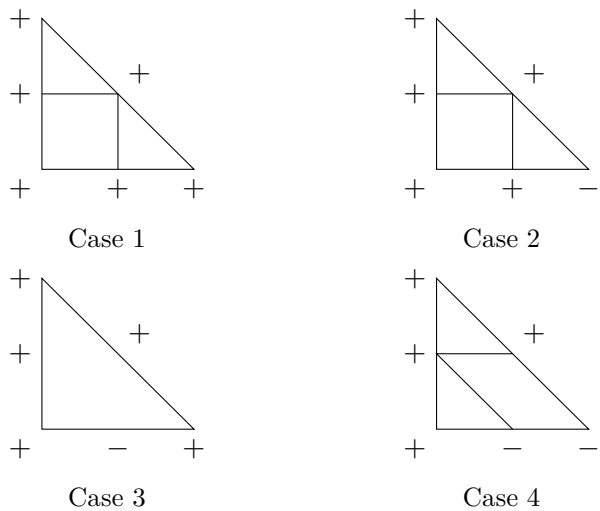
\begin{figure}
        \centering
            \begin{subfigure}{0.4\textwidth}
            \centering
                \begin{tikzpicture}
                    \draw (0,0) -- (0,2) -- (2,0) -- cycle;
                    \draw (0,1) -- (1,1) -- (1,0);
                    \foreach \i in {1, 2} {
                        \draw (0,\i) node[left] {$+$};
                        \draw (\i,0) node[below] {$+$};
                    }
                    \draw (0,0) node[below left] {$+$};
                    \draw (1,1) node[above right] {$+$};
                \end{tikzpicture}
                \caption*{Case 1}
            \end{subfigure}
            \begin{subfigure}{0.4\textwidth}
            \centering
                \begin{tikzpicture}
                    \draw (0,0) -- (0,2) -- (2,0) -- cycle;
                    \draw (0,1) -- (1,1) -- (1,0);
                    \foreach \i in {1, 2} {
                        \draw (0,\i) node[left] {$+$};
                    }
                    \draw (1,0) node[below] {$+$};
                    \draw (2,0) node[below] {$-$};
                    \draw (0,0) node[below left] {$+$};
                    \draw (1,1) node[above right] {$+$};
                \end{tikzpicture}
                \caption*{Case 2}
            \end{subfigure}
            \begin{subfigure}{0.4\textwidth}
            \centering
                \begin{tikzpicture}
                    \draw (0,0) -- (0,2) -- (2,0) -- cycle;
                    \foreach \i in {1, 2} {
                        \draw (0,\i) node[left] {$+$};
                    }
                    \draw (1,0) node[below] {$-$};
                    \draw (2,0) node[below] {$+$};
                    \draw (0,0) node[below left] {$+$};
                    \draw (1,1) node[above right] {$+$};
                \end{tikzpicture}
                \caption*{Case 3}
            \end{subfigure}
            \begin{subfigure}{0.4\textwidth}
            \centering
                \begin{tikzpicture}
                    \draw (0,0) -- (0,2) -- (2,0) -- cycle;
                    \draw (1,0) -- (0,1) -- (1,1);
                    \foreach \i in {1, 2} {
                        \draw (0,\i) node[left] {$+$};
                    }
                    \draw (1,0) node[below] {$-$};
                    \draw (2,0) node[below] {$-$};
                    \draw (0,0) node[below left] {$+$};
                    \draw (1,1) node[above right] {$+$};
                \end{tikzpicture}
                \caption*{Case 4}
            \end{subfigure}
        \caption{The 4 cases of Newton-degree $2$ sign configurations and subdivisions.}
        \label{fig:4cases}
    \end{figure}

    We now need to show that $\pmult^\S_l(f)=\bmult^\S_l(f)$ for all Newton-degree-$1$ polynomials $l\in \S[x,y]$. After scaling, we may assume that $l=1+sx+ty$ for some $s,t\in \S^*$. In all four cases, the constant, the $y$, and the $y^2$ coefficient are positive, so $\bmult^S_l(f)=0$ unless $t=1$. In case $1$, we have $\bmult^S_l(f)=0$ if $s=-1$ and 
    \[
\bmult^S_l(f)=2=\pmult_l^\S(f)
    \]
    if $s=+1$, where the subdivision realizing the perturbation multiplicity is depicted in Figure~\ref{fig:4cases}. In case 3, we have $\bmult_l^\S(f)=0$ for any choice of $s$. In cases 2 and 4, we have
    \[
    \bmult_l(f)=1=\pmult_l(f) 
    \]
    for all $s\in \S^*$, where the subdivision realizing the perturbation multiplicity is depicted in Figure~\ref{fig:4cases} (the same subdivision works for both choices of $s$).
\end{proof}

\begin{example}
    In dimension at least $3$, there exist dense quadratic polynomials with $\mult^\sgn_{1+\sum x_i}(f)<\mult^\S_{1+\sum x_i}(f)$. To see this, consider the polynomial 
    \[
    f=1+x+y-z-xy-xz+yz-x^2+y^2-z^2\in \S[x,y,z] . 
    \]
    Let $l=1+x+y+z$. Then we check that
    \[
    f\in (1+x+y+z)(1-x+y-z)
    \]
    over $\S$ and hence $\mult^\S_l(f)=1$. Now assume $\mult^\sgn_l(f)=1$. Then there exist polynomials 
    $g, h\in \R[x,y,z]$ with $(g\cdot h)^\sgn= f$ and $g^\sgn=l$. After first scaling $g$ such that its constant coefficient is $1$ and then rescaling each variable, we may assume that $g=1+x+y+z$.
    Write $h=a+bx+cy+dz$ for $a,b,c,d\in \R$. Looking at the coefficients of $x$, $xy$, $yz$, and $z$ in $gh$ we obtain the inequalities
    \begin{align*}
        a+b &> 0 \\
        b+c &<0 \\
        c+d &>0 \\
        a+d &<0 , \\
    \end{align*}
    which leads to the contradiction
    \[
    a>-b>c>-d>a . \qedhere
    \]
\end{example}

\begin{theorem}
Let $f\in \S[x,y]$ be a (not necessarily dense) polynomial of Newton-degree $2$, and let $l\in \S[x,y]$ be of Newton-degree $1$. Then we have
\[
\mult^\sgn_l(f)=\mult_l^\S(f)= \bmult_l^\S(f) .
\]
\end{theorem}

\begin{proof}
    By Theorem~\ref{thm:multiplicity for dense quadratic poly in 2 variables} we only need to treat the cases where $f$ is not dense, and by Proposition~\ref{prop:comparison relative multiplicity and hyperfield multiplicity} and Corollary~\ref{cor:mult leq bmult} is suffices to show that 
    \[
    \mult^\sgn_l(f) = \bmult_l^\S(f)
    \]
    If a coefficient of a middle term (e.g.\ $x$) in $f$ is zero, then $\bmult_l^\S(f)$ is zero unless the coefficients of the adjacent corners of the Newton polytope (e.g.\ $1$ and $x^2$) have different signs. Therefore, if all three middle terms of $f$ are zero, we have $\bmult_l^\S(f)=0$. We may thus assume that either one or two middle terms are zero. After interchanging variables (as in the proof of Theorem~\ref{thm:multiplicity for dense quadratic poly in 2 variables}), we may assume that either only the $x$-coefficient is zero or the $x$- and $y$-coefficient are both zero. After scaling $f$ by a unit, we may assume that the constant coefficient is $1$, in which case we may assume that the $x^2$-coefficient is $-1$. If the $y$-coefficient is also zero, we may also assume that the $y^2$-coefficient is $-1$. Using the transformations $x\leftrightarrow -x$ and $y\leftrightarrow -y$ we may assume that the non-zero middle terms have coefficient $1$. This leaves us with three cases for $f$, as depicted in Figure~\ref{fig:3cases non-dense quadratic}. 
    After rescaling $l$, we may assume that the constant coefficient of $l$ is $1$. 
    
    In case 1, we have $\bmult_l^\S(f)=0$ unless $l=1\pm x+y$, in which case $\bmult_l^\S(f)=1$. We also have
    \[
    f=\big( (1+ x+y)(1-x+2y) \big)^\sgn.
    \]
    This shows that $\pmult_l^\S(f)=1$ for either choice of $l$. 

    In case 2, we have $\bmult_l^\S(f)=0$ unless $l=1\pm x \mp y$, in which case $\bmult_l^\S(f)=1$. We also have
    \[
    f=\big( (1+ x-y)(1-x+2y) \big)^\sgn.
    \]
    This shows that $\pmult_l^\S(f)=1$ for either choice of $l$.

    In case 3, we have $\bmult_l^\S(f)=0$ unless $l=1\pm x \mp y$, in which case $\bmult_l^\S(f)=1$. We also have
    \[
    f=\big( (1+ x-y)(1-x+y) \big)^\sgn.
    \]
    This shows that $\pmult_l^\S(f)=1$ for either choice of $l$.
\end{proof}

    \begin{figure}
        \centering
            \begin{subfigure}{0.3\textwidth}
            \centering
    \begin{tikzpicture}[baseline=(current bounding box.center)]
        \matrix[matrix of math nodes]{
        + &   &     \\
        + & + &     \\
        + & 0 & -   \\
        }; 
    \end{tikzpicture}
                \caption*{Case 1}
    \end{subfigure}
            \begin{subfigure}{0.3\textwidth}
            \centering
    \begin{tikzpicture}[baseline=(current bounding box.center)]
        \matrix[matrix of math nodes]{
        - &   &     \\
        + & + &     \\
        + & 0 & -   \\
        }; 
    \end{tikzpicture}
    \caption*{Case 2}
     \end{subfigure}
            \begin{subfigure}{0.3\textwidth}
            \centering
    \begin{tikzpicture}[baseline=(current bounding box.center)]
        \matrix[matrix of math nodes]{
        - &   &     \\
        0 & + &     \\
        + & 0 & -   \\
        }; 
    \end{tikzpicture}
                \caption*{Case 3}
            \end{subfigure}
        \caption{The 3 non-dense cases needed to be checked after all reductions.}
        \label{fig:3cases non-dense quadratic}
    \end{figure}

\begin{example}
If $f\in \S[x,y]$ is quadratic but not dense, and $l\in \S[x,y]$ has degree $1$, it is possible that $\pmult^\S_l(f)< \mult^\S_l(f)$. For example, consider the polynomial
\[
f(x,y)= 1-x^2+xy-y^2\in \S[x,y]
\]
and let 
\[
l(x,y)=1+x-y .
\]
Then we have $\mult^\S_l(f)=\bmult^\S_l(f)=1$. On the other hand, the only subdivision of the Newton polytope of $f$ that appears as the Newton subdivision of a strictly convex polynomial in $\angular^{-1}\{l\}$ is depicted in Figure~\ref{fig:only Newton subdivision}.
Since the tropical hypersurface associated to any polynomial $h\in\TR[x,y]$ with that Newton subdivision can never contain a tropical line, we have
$\gmult^\K_{\angular^{-1}\{l\}}(h)=0$ and hence $\mult^{\TR}_{\angular^{-1}\{l\}}(h)=0$ by Lemma~\ref{lem:multipliciy bounded by geometric multiplicity}. In particular, we have 
\[
\pmult_l^\S(f)=0<1=\mult^\S_l(f). \qedhere
\]
\end{example}

\begin{figure}
    \centering
    \begin{tikzpicture}
                    \draw (0,0) -- (0,2) -- (2,0) -- cycle;
                    \draw (0,0) -- (1,1);
                    \draw (2,0) node[below] {$-$};
                    \draw (0,2) node[left] {$-$};
                    \draw (0,0) node[below left] {$+$};
                    \draw (1,1) node[above right] {$+$};
                \end{tikzpicture}
    \caption{The only Newton subdivision including the support of $1-x^2+xy-y^2$ as vertices.}
    \label{fig:only Newton subdivision}
\end{figure}
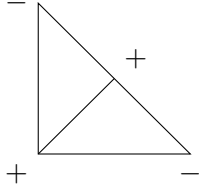

\section{Systems of equations over hyperfields}
\label{sec:systems}

Let $K$ be a field with a morphism $\varphi\colon K\to H$ to a hyperfield $H$, let $f_1,\ldots, f_n\in H[x_1,\ldots,x_n]$, and let $\bm h\in (H^*)^n$. In this section, we study the number
\[
N_{\bm h}^\varphi(f_1,\ldots, f_n)=\max\left\{\left\vert\bigcap V(g_i)\cap \varphi^{-1}\{\bm h\}\right\vert : g_i^\varphi=f_i,~\left\vert\bigcap V(g_i)\right\vert<\infty\right\}  .
\]
In the case where $H=\S$ (resp.\ $H=\T$), this is the maximum number of solutions with given signs (resp.\ given valuations) that a system of equations with given supports and signs (resp.\ valuations) can have, provided it has finitely many solutions.
Our technique to bound this number is via sparse resultants, which translate the problem of finding solutions to a system of equations into the problem of finding linear factors of a single multivariate polynomial.

\subsection{Sparse resultants}

Let $A_0,\ldots, A_n$ be subsets of $\Z_{\geq 0}^n$. For each $0\leq i\leq n$ and $\bm a\in A_i$ introduce a variable $c_{i,\bm a}$. Then the (sparse mixed) resultant $R=\result{A_0,\ldots, A_n}$ of $A_0,\ldots, A_n$ is the unique (up to scaling) irreducible integer polynomial in the variables $c_{i,\bm a}$, which vanishes precisely when the intersection
\begin{equation}
\label{eq:intersection of n+1 hypersurfaces}
\bigcap_{i=0}^n V\left(\sum_{\bm a\in A_i}c_{i,\bm a} \bm x^{\bm a}\right)\cap (K^*)^n 
\end{equation}
is nonempty for some (and hence any) algebraically closed field $K$ of characteristic $0$.
We expect the intersection to be nonempty on a codimension $1$ set because there is one more equation than variables ($x_1,\dots,x_n$). Only if the codimension is indeed $1$ the resultant is well-defined; otherwise one sets $R=1$. For more on resultants, we refer the reader to the book of Gelfand-Kapranov-Zelevinsky~\cite{GKZ}. The resultants we use here are the mixed $(A_0, \dots, A_n)$-resultants covered in Chapter~8 of their book.

Given $n+1$ polynomials in $n$-variables, say $g_i=\sum_{\bm a\in A_i} d_{i,\bm a} \bm x^{\bm a}\in H[\bm x]$ for $0\leq i\leq n$ over some hyperfield $H$, we denote by $\result{g_0,\ldots, g_n}$ the set (we get a set because hyperaddition is multivalued) of polynomials obtained by substituting $d_{i,\bm a}$ for $c_{i,\bm a}$ in $\result{A_0,\ldots, A_n}$. If only $n$ polynomials in $n$ variables are given, say the polynomials $g_1,\ldots, g_n$ with the expressions as before, we introduce new variables $y_1, \dots, y_n$ and set 
\[
\result{g_1,\ldots, g_n}= \result{1+\sum y_ix_i,g_1,\ldots,g_n}\subseteq H[\bm y],
\]
substituting $y_i$ for the variables $c_{0,\bm e_i}$ corresponding to 
\[
A_0=\{0\}\cup \{\bm e_i :  1\leq i\leq n\},
\]
where $\bm e_i$ denotes the $i$-th standard basis vector in $\Z_{\geq 0}^n$.

The fact that resultants translate the problem of finding solutions to systems of equations to the problem of finding linear factors of a polynomial already mentioned above, is made precise in the following lemma.

\begin{lemma}
\label{lem:factors of resultant}
    Let $K$ be a field of characteristic $0$ and let $\varphi\colon K\to H$ be a morphism of hyperfields. Moreover, let $\bm h\in (H^*)^n$, let $l=1+\sum_{i=1}^n h_ix_i\in H[\bm x]$, and let $g_1,\ldots, g_n\in K[\bm x]$ generic with respect to their support and such that $R=\result{g_1,\ldots, g_n}$ is not constant. Then we have
    \[
      \left\vert \bigcap_{i=1}^n V(g_i) \cap \varphi^{-1}\{\bm h\}  \right\vert = \mult^K_{\varphi^{-1}\{l\}}(R) .
    \]
\end{lemma}

\begin{proof}
Because the coefficients of the $g_i$ are generic with respect to their supports, the intersection
\[
\bigcap_{i=1}^n V(g_i) 
\]
is transverse and consists of $D\coloneqq \deg(R)$ many distinct points
\[
p_j=(p_{j1},\ldots p_{jn})\in (\overline K^*)^n , \qquad 1\leq j\leq D.
\]
 Then the intersection
\[
\bigcap_{i=1}^n V(g_i) \cap V\left(1+\sum_{i=1}^n y_ix_i\right)
\]
is nonempty if and only if
\[
1+\sum_{i=1}^n p_{ji}y_i =0
\]
for some $1\leq j\leq D$, which happens, by definition of the resultant, if and only if
\[
R(y_1,\ldots, y_n)=0 .
\]
Because $D$ is the degree of $R$, it follows that $R$ differs from 
\[
\prod_{j=1}^D \left(1+\sum_{i=1}^n p_{ji}y_i\right)
\]
by a unit. The assertion now follows from the observation that $\varphi(p_j)= \bm h$ if and only if
\[
\left( 1+\sum_{i=1}^n p_{ji}y_i\right)^\varphi = l . \qedhere
\]
\end{proof}

An important observation in the proof of the preceding lemma is that a resultant $\result{g_1, \dots, g_n}$ is (up to a unit), the product of the linear forms $1 + \sum p_{ji} y_i$ corresponding to the common roots $p_j$ of the system
\[
g_1(\bm x)=\ldots= g_n(\bm x)=0 
\]
in the algebraic closure of the ground field.
Let us illustrate this with an example.

\begin{example}
    Take the line $f(x, y) = 3x + 4y - 5$ and intersect it with the circle $g(x, y) = x^2 + y^2 - 1$. These two polynomials have one intersection point $[3 : 4 : 5] \in \P^2$, with multiplicity $2$. The resultant of $f$ and $g$ in the variables $u, v$ is therefore proportional to $(3u + 4v + 5)^2$.

    We can compute this in the Singular computer algebra system~\cite{sing} using the \texttt{mpresmat} function.
    \begin{verbatim}
    system("random", 12341234);
        // other seeds lead to different monomial factors
    ring R = 0,(u,v),dp;
    ring S = R,(x,y),dp;
    ideal I = 3x + 4y - 5, x2 + y2 - 1, 1 + ux + vy;
    string s = string(det(mpresmat(I, 0)));
        // use a string to get this polynomial from S to R
        // s = (9u2+24uv+30u+16v2+40v+25)
    setring R;
    execute("poly p = " + s);
    factorize(p);
        // Output (factors and multiplicities)
        // [1]:
        //    _[1]=1
        //    _[2]=3u+4v+5
        // [2]:
        //    1,2
    \end{verbatim}
    \vspace*{-\dimexpr2\baselineskip + \topsep + \partopsep}
\end{example}

\subsection{Tropically transverse intersections}

We will now study the cases where $H=\T$ or $H=\TR$, where $\varphi\colon K\to H$ is either a valuation $\nu$ or a signed valuation $\vsign$, and where the intersection 
\[
\bigcap_{i=1}^n V(f_i^\nu)
\]
in $\R^n$ is transverse. Recall that this means that $\bigcap^n_{i=1}V(f_i^\nu)$ is finite and every $\bm h \in \bigcap_{i=1}^n V(f_i^\nu)$ is contained in the relative interior of a maximal cell of $V(f_i^\nu)$ for all $1\leq i \leq n$.

For every choice of $g_i\in \varphi^{-1}\{f_i\}$ and $\bm h\in \bigcap_{i=1}^n V(g_i)\cap (K^*)^n$ we then have $\varphi(\bm h)\subseteq \bigcap V(f_i^\nu)$. Therefore, we have
\[
N_{\bm h}^\varphi(f_1,\ldots, f_n)=0 
\]
for all $\bm h\notin \bigcap_{i=1}^n V(f_i^\nu)$.  

Now suppose $\bm h\in \bigcap_{i=1}^n V(f_i^\nu)$. Then for every $1\leq i\leq n$, the initial form $\initial_{\bm h}(f_i)$ is a binomial, say $f_i=a_i\bm x^{\bm s_i}-b_i\bm x^{\bm t_i}$. We define the  \emph{intersection multiplicity} $m^\K(\bm h; f_1^\nu\cdots f_n^\nu))$  as
\[
m^\K(\bm h; f_1^\nu\cdots f_n^\nu)= \left\vert\left(
\begin{array}{c}
\bm s_1 - \bm t_1  \\ \hline
\vdots \\ \hline
\bm s_n - \bm t_n
\end{array}
\right)
\right\vert.
\]

\begin{lemma}[{\cite[Lemma 3.2]{HuberSturmfels}}]
\label{lem:solutions to binomial system}
    Let $\bm h\in \bigcap_{i=1}^n V(f_i^\nu)$ and for $1\leq i\leq n$ let $g_i$ be polynomials with $g_i^{\nu_0}=\initial_{\bm h}(f_i)^{\nu_0}$ over an algebraically closed field of characteristic $0$. Then $\bigcap_{i=1}^n V(g_i)$ contains precisely $m^\K(\bm h;f_1\cdots f_n)$ many distinct points.
\end{lemma}

Now suppose that $f_i\in \TR[\bm x]$, and still assume that $V(f_1^\nu),\ldots,V(f_n^\nu)$ intersect transversally. Let $\bm h\in \nu^{-1}\bigcap_{i=1}^n V(f_i^\nu) \subseteq (\TR^*)^n$. Then $\initial_{\nu(\bm h)}(f_i)$ is a binomial for all $1\leq i\leq n$. Following \cite{IR}, we say that $\bm h$ is \emph{alternating}\label{def:alternating} if the two coefficients of the binomial $\initial_{\nu(\bm h)}(f_i)$ have opposite signs for all $1\leq i\leq n$. If $\angular(\bm h)=(1,\ldots,1)$, we define the signed multiplicity $m^\S(\bm h;f_1\cdots f_n)$ by
\[
m^\S(\bm h;f_1\cdots f_n)= 
\begin{cases}
    1 & \text{if }\bm h\in \bigcap_{i=1}^n V(f_i^\nu) \text{ and }\bm h \text{ is alternating}, \\
    0 & \text{else.}
\end{cases}
\]
For general $\bm h$, let $\vert \bm h\vert= (\angular(h_1)h_1,\ldots, \angular(h_n)h_n)$ and for $1\leq i\leq n$ denote 
\[
f_i^{\bm h}(x_1,\ldots,x_n)=f_i(\angular(h_1)x_1,\ldots, \angular(h_n)x_n) ,
\]
where we identify $\S^*$ with $\nu^{-1}\{0\} = \{\pm t^0\} \subseteq\TR$. The signed multiplicity is then given by
\[
m^\S(\bm h;f_1\cdots f_n)= m^\S(\vert \bm h\vert; f_1^{\bm h}\cdots f_n^{\bm h}).
\]

\begin{lemma}[{\cite[Lemma 2]{IR}}]
\label{lem:at most one positive root of binomial system}
Let $K$ be a real closed field. Suppose we have binomials $g_1,\ldots, g_n\in K[\bm x]$ such that the affine span of all the Newton polytopes of the $g_i$ is $\R^n$. If for some $1\leq i\leq n$, the coefficients of the two monomials of $g_i$ have the same sign, then the intersection
\[
\bigcap_{i=1}^n V(g_i) \cap (K_{>0}) ^n
\]
is empty. Otherwise, it is a singleton.

In particular, suppose $f_1,\ldots, f_n\in \TR[\bm x]$ and $\bm h\in (\TR^*)^n$ are such that $V(f_1^\nu),\ldots, V(f_n^\nu)$ intersect transversally at $\nu(\bm h)$. If $g_i\in \sign^{-1}\{\initial_{\nu(\bm h)}(f_i)\}$, then we have
\[
\left\vert\bigcap_{i=1}^n V(g_i) \cap \sign^{-1}\{\angular(\bm h)\}\right\vert= m^\S(\bm h;f_1\cdots f_n) .
\]
\end{lemma}

\begin{proof}
    The statement about the positive common roots of the $g_i$ is proven in \cite[Lemma 2]{IR}. The ``in particular'' statement follows directly from that in the case where $\angular(\bm h)=(1,\ldots, 1)$. The general case is reduced to that case by the coordinate change $x_i\mapsto \angular(h_i)x_i$.
\end{proof}

We have the following relationship between the initial form of a resultant and the resultant of initial forms.

\begin{proposition}
\label{prop:initial of resultant}
Let $(K,\nu)$ be a valued field of characteristic $0$, equipped with a splitting of the valuation, and let $g_i\in K[\bm x]$ for $1\leq i\leq n$. Assume that $V(g_1^\nu),\ldots, V(g_n^\nu)$ intersect transversally at $\bm h\in \R^n$. Then $\initial_{-\bm h}\result{g_1,\ldots,g_n}$ and 
    $\result{\initial_{\bm h}(g_1),\ldots, \initial_{\bm h}(g_n)}$
differ by a polynomial $q$ with \[\mult_{\nu_0^{-1}\{1+\sum_{j=1}^n x_j\}}(q)=0. \]
\end{proposition}

\begin{proof}
For $1\leq i \leq n$ denote the support of $g_i$ by $A_i$ and let 
\[
A_0=\{0\} \cup \{\bm e_i: 1\leq i\leq n\},
\]
where $\bm e_i$ denotes the $i$-th standard basis vector.
Moreover, let $R=\result{A_0,\ldots, A_n}$ be the resultant of the supports, which is a polynomial in coefficients $c_{i,\bm a}$, where $0\leq i\leq n$ and $\bm a \in A_i$. We defined $\result{g_1,\ldots, g_n}$ as a polynomial in variables $y_1,\ldots, y_n$, but in this proof we will substitute $c_{0,\bm e_i}$ for $y_i$ and view $\result{g_1,\ldots, g_n}$ as a polynomial in the variables $c_{0,\bm e_1},\ldots, c_{0,\bm e_n}$. Then $\result{g_1,\ldots, g_n}$ is obtained by plugging $1$ for $c_{0,\bm 0}$ and $d_{i, \bm a}$ for $c_{i, \bm a}$ for $i>0$ and $\bm a \in A_i$ into $R$. We note that $R$ is homogeneous in the coefficients $c_{0,\bm 0},c_{0,\bm e_1},\ldots c_{0,\bm e_n}$, so plugging in $1$ for $c_{0,\bm 0}$ amounts to dehomogenizing. Therefore, $\initial_{-\bm h}(\result{g_1,\ldots,g_n})$ is equal to the polynomial we obtain by plugging in $1$ for $c_{0,\bm 0}$ into the initial form
\[
\initial_{(0,-\bm h)}R\left((c_{0,\bm a})_{\bm a \in A_0} , (d_{i,\bm a})_{i>0, \bm a\in A_i}\right),
\]
where the additional $0$ in $(0,\bm h)$ means that we give $c_{0,\bm 0}$ weight zero.
Let $\bm w=(0,-\bm h,(\nu(d_{i, \bm a}))_{i>0,\; \bm a \in A_i})$. We view $\bm w$ as a weight on $\R^{\bigsqcup_{i=0}^n A_i}$. If for a monomial $M$ of $R$, we denote $M'=M((c_{0,\bm a})_{\bm a \in A_0}, (d_{i, \bm a})_{i>0,\bm a \in A_i})$, then the $\bm w$-weight of $M$ with respect to the trivial valuation $\nu_0$ equals the $(0,-\bm h)$-weight of $M'$ with respect to $\nu$ (note that $R$ has integer coefficients). It follows that if 
\[
(\initial^0_{\bm w}(R))((c_{0,\bm a})_{\bm a \in A_0}, (\angular(d_{i, \bm a}))_{i>0,\; \bm a \in A_i})\neq 0,
\]
where the superscript $0$ in $\initial^0$ indicates that we take the initial form with respect to the trivial valuation,  then we have
\begin{align*}
&\initial_{(0,-\bm h)}(R( (c_{0,\bm a})_{\bm a \in A_0}, (d_{i, \bm a})_{i>0,\; \bm a \in A_i})) \\
&\quad= (\initial^0_{\bm w}(R))((c_{0,\bm a})_{\bm a \in A_0}, (\angular(d_{i, \bm a}))_{i>0,\; \bm a \in A_i}) .
\end{align*}
To finish the proof, we compute $(\initial^0_{\bm w}(R))((c_{0,\bm a})_{\bm a \in A_0}, (\angular(d_{i, \bm a}))_{i>0,\; \bm a \in A_i})$ and, in particular, show that it is non-zero.
To this end, let $g_0=1+\sum_{i=1}^n c_{0,\bm e_i}x_i$, and let $\Delta$ be the polyhedral complex in $\R^n$, the relative interior of whose faces are precisely the equivalence classes of the relation
\[
\bm w_1\sim \bm w_2 \Longleftrightarrow \initial_{\bm w_1}(g_i)=\initial_{\bm w_2}(g_i)\text{ for all }0\leq i\leq n .
\]
Here, we give weight $-h_i$ to the coefficient $c_{0, \bm i}$ of $x_i$ in $g_0$. Note that  $\Delta$ coincides with the intersection of the $n+1$ complexes on $\R^n$ induced by the tropical hypersurfaces $V(g_i^\nu)$. 
By \cite[Theorem 4.1]{sturmfels-resultants}, we have 
\[
    (\initial^0_{\bm w}(R))((c_{0,\bm a})_{\bm a \in A_0}, (\angular(d_{i, \bm a}))_{i>0,\; \bm a \in A_i})=\pm\prod_{\bm v} R_{\bm v}^{d_{\bm v}} ,
\]
where the product runs over all vertices $\bm v$ of $\Delta$, and where
\[
R_v=\result{\initial_{\bm v}(g_0) , \initial_{\bm v} (g_1),\ldots \initial_{\bm v}(g_n)} 
\]
and the $d_v$ are positive integers that can be computed explicitly in terms of the supports of the $\initial_v(g_i)$.

The resultant $R_{\bm v}$ is a monomial if at least one of the $\initial_{\bm v}(g_i)$ is a monomial. Therefore, the set of vertices $\bm v$ for which $R_{\bm v}$ is not a monomial is contained in the set $S$ defined by $S=\bigcap_{i=0}^n V(g_i^\nu)$. 
For each $\bm v\in S$ the polynomials $\initial_{\bm v}(g_i)$ for $1\leq i\leq n$ are binomials that intersect in finitely many points, by Lemma~\ref{lem:solutions to binomial system}, no matter how we vary their coefficients.
Therefore, $R_{\bm v}\neq 0$. Moreover, for $\bm h\neq \bm v \in S$ the initial form $\initial_{\bm v}(g_0)$ has support strictly smaller than the support of $g_0$. As $R_{\bm v}$ is a product of polynomials with the same support as $\initial_{\bm v}(g_0)$, this implies that
\[
\mult_{\nu_0^{-1}\{1+\sum_{j=1}^n x_j\}}(R_{\bm v})=0 .
\]
Finally, according to \cite[Theorem 4.1]{sturmfels-resultants} we have $d_{\bm h}=1$ because $\initial_{\bm h}(g_0)$ and $g_0$ have the same support and the support of  $g_0$ spans $\Z^n$. 
\end{proof}

\begin{theorem}
\label{thm:intersection multiplicity in transversal case}
Let $K$ be an algebraically closed valued field or  a real closed valued field with compatible valuation, with residue field $\kappa$. Let $H=\kappa/\kappa^2$ (either $\K$ or $\S$). Let $\overline\varphi\colon \kappa\to H$ denote the quotient morphism, and let $\varphi\colon K\to H \rtimes \R$ denote the composite $K\xrightarrow{\vangular} \kappa \rtimes \R \xrightarrow{\varphi\rtimes \R} H \rtimes \R$. Furthermore, let $f_1,\ldots, f_n\in (H \rtimes \R)[\bm x]$ be such that $V(f_1^\nu),\ldots, V(f_n^\nu)$ intersect transversally, and let $\bm h\in ((H \rtimes \R)^*)^n$. Then we have
\[
    N_{\bm h}^\varphi(f_1,\ldots, f_n)=m^H(\bm h;f_1\cdots f_n) .
\]
In fact, for every generic choice of $g_i\in \varphi^{-1}\{f_i\}$ for $1\leq i\leq n$ we have
\[
   \left\vert \bigcap_{i=1}^n V(g_i) \cap \varphi^{-1}\{\bm h\}\right\vert= m^H(\bm h;f_1\cdots f_n) .
\]
\end{theorem}

\begin{remark}
    If $K$ is algebraically closed, then $H \rtimes \R=\T$ and $\varphi=\nu$, and if $\K$ is real closed, then $H \rtimes \R=\TR$ and $\varphi=\vsign$.
\end{remark}

\begin{proof}
For $1\leq i\leq n$ let $g_i\in \varphi^{-1}\{f_i\}$, let $R=\result{g_1,\ldots, g_n}$, and let $l=1+\sum_{i=1}^n h_ix_i\in \TR[\bm x]$. By Lemma~\ref{lem:factors of resultant}, we have 
\[
   \left\vert \bigcap_{i=1}^n V(g_i) \cap \varphi^{-1}\{\bm h\}\right\vert=\mult^K_{\varphi^{-1}\{l\}}(R).
\]
 By Proposition~\ref{prop:initial ideal gives multiplicity for product of linear forms} in the algebraically closed case and Lemma~\ref{lem:solutions to binomial system} and Proposition~\ref{prop:initial ideal gives multiplicity for product of linear forms, real case} in the real closed case, we have
 \[
    \mult^K_{\varphi^{-1}\{l\}}(R)=\mult^\kappa_{\overline\varphi^{-1}\{\initial_{-\nu(\bm h)}(l)\}}(\initial_{-\nu(\bm h)}(R)) .
 \]
 By Proposition~\ref{prop:initial of resultant}, we have 
 \begin{align*}
    &\mult^\kappa_{\overline\varphi^{-1}\{\initial_{-\nu(\bm h)}(l)\}}(\initial_{-\nu(\bm h)}(R)) \\
    &\quad= \mult^\kappa_{\overline\varphi^{-1}\{\initial_{-\nu(\bm h)}(l)\}}(\result{\initial_{\nu(\bm h)}(g_1),\ldots, \initial_{\nu(\bm h)}(g_n)}) ,
\end{align*}
which, again by Lemma~\ref{lem:factors of resultant}, is equal to
 \[
    \left\vert\bigcap_{i=1}^n V(\initial_{\nu(\bm h)}(g_i))\cap  \overline\varphi^{-1}\{\angular(\bm h)\}\right\vert . 
 \]
 By Proposition~\ref{prop:initial ideal gives multiplicity for product of linear forms} in the algebraically closed case and Proposition~\ref{prop:initial ideal gives multiplicity for product of linear forms, real case} in the real closed case, we have
 \[
    \left\vert\bigcap_{i=1}^n V(\initial_{\bm h}(g_i) \cap \overline\varphi^{-1}\{\angular(\bm h)\}\right\vert= m^H(\bm h;f_1^\nu\cdots f_n^\nu) . \qedhere
 \]
\end{proof}

Using some model theory, we can now use our results about the numbers $N^\vsign_{\bm h}(f_1,\ldots,f_n)$ for $f_i\in \TR[\bm x]$ to obtain the following result about the analogous numbers for $f_i\in \S[\bm x]$. As further explained below after Definition~\ref{def:perturbation multiplicity for systems}, we reprove the main Corollary to \cite[Theorem 2]{IR}.

\begin{corollary}
\label{cor:Itenberg roy lower bound}
    Let $K$ be a real closed field and let $f_1,\ldots, f_n\in \TR[\bm x]$ such that the tropical hypersurfaces $V(f_i^\nu)$ intersect transversally. Moreover, let $\bm h\in (\S^*)^n$ and denote
    \[
    G= \angular^{-1}\{\bm h\} \cap \nu^{-1}\left(\bigcap_{i=1}^n V(f_i^\nu)\right)\subseteq (\TR^*)^n .
    \]
    Then we have
    \[
    N^\sign_{\bm h}(f_1^\angular,\ldots, f_n^\angular)\geq \sum_{\bm g\in G} m^\S(\bm g;f_1\cdots f_n).
    \]
\end{corollary}

\begin{proof}
    First, note that the inequality
    \[
    N^\sign_{h}(f_1^\angular,\ldots, f_n^\angular)\geq \sum_{\bm g\in G} m^\S(\bm g;f_1\cdots f_n)
    \]
    can be formulated in the language of real closed fields. Since the theory of real closed fields is complete (see e.g.\ \cite[Chapter~3.3]{model-theory}), we may assume that $K$ is a valued real closed field with surjective valuation. We pick, for $1\leq i \leq n$, a polynomial $g_i\in K[\bm x]$ with $g_i^\vsign=f_i$.  Then we have
    \begin{multline*}
    N^\sgn_{\bm h}(f_1^\angular,\ldots, f_n^\angular) \geq 
    \left\vert\bigcap_{i=1}^n V(g_i) \cap \sgn^{-1}\{\bm h\}\right\vert=\\
    = \sum_{\bm g\in G}\left\vert \bigcap_{i=1}^n V(g_i) \cap \vsign^{-1}\{\bm g\}  \right\vert 
    = \sum_{\bm g\in G} m^\S(\bm g; f_1\cdots f_n) ,    
    \end{multline*}
    where the last equality follows from Theorem~\ref{thm:intersection multiplicity in transversal case}.
\end{proof}

\begin{definition}
\label{def:perturbation multiplicity for systems}
Let $f_1,\ldots, f_n \in \S[\bm x]$, let $h\in (\S^*)^n$, and let $\widetilde F$ be the sets of tuples $(\widetilde f_1,\ldots, \widetilde f_n)$ of polynomials $\widetilde f_i\in \TR[\bm x]$ with $\widetilde f_i^\angular= f_i$ and such that $V(\widetilde f_1^\nu), \dots, V(\widetilde f_n^\nu)$ intersect transversally. In analogy to the perturbation multiplicity, we define
\[
\sN_{\bm h}(f_1,\ldots,f_n)= \max\left\{ \sum_{\bm g\in G(\bm h;\widetilde f_1,\ldots, \widetilde f_n)} m^\S(\bm g;\widetilde f_1\cdots \widetilde f_n) ~ : ~ (\widetilde f_i)_i\in  \widetilde F\right\} ,
\]
where 
\[
G(\bm h;\widetilde f_1,\ldots, \widetilde f_n) =\angular^{-1}\{\bm h\}\cap \nu^{-1}\left(\bigcap_{i=1}^n V(\widetilde f_i^\nu)\right) .
\]
\end{definition}

The statement of Corollary~\ref{cor:Itenberg roy lower bound} can now be rephrased as
\begin{equation}
\label{eq:Itenberg Roy bound}
N_{\bm h}^\sgn(f_1,\ldots, f_n)\geq \sN_{\bm h}(f_1,\ldots, f_n) .    
\end{equation}
If we identify $f_i\in \S[\bm x]$ with its signed Newton polytope and $\bm h$ with the orthant of $\R^n$ it determines, then the number $\sN_{\bm h}(f_1,\ldots,f_n)$ is precisely what is denoted by $n((f_1,\ldots, f_n), \bm h)$ by Itenberg-Roy~\cite{IR}. Corollary~\ref{cor:Itenberg roy lower bound} follows from \cite[Theorem 2]{IR}. Based on the inequality \eqref{eq:Itenberg Roy bound} and the idea that the tropically transverse case is the most degenerate and therefore that with the most real solutions, Itenberg and Roy conjectured [\emph{loc.\ cit.}] that there is equality in \eqref{eq:Itenberg Roy bound}. This was later disproven by Li and Wang with an explicit counterexample \cite{LW}. We will revisit that counterexample below in Example~\ref{ex:Li Wang counterexample}.

\subsection{Resultants over hyperfields}

As before, let $f_1,\ldots, f_n\in H[\bm x]$, where $f_i=\sum_{\bm a \in A_i}d_{i, \bm a} \bm x^{\bm a}$, let $h\in (H^*)^n$, and let $\varphi\colon K\to H$ be a morphism from a field $K$ to $H$. We wish to give an upper bound for
\[
N^\varphi_{\bm h}(f_1,\ldots, f_n)
\]
in terms of the multiplicities introduced in the previous section. Recall that $\result{f_1,\ldots, f_n}$ denotes the set of polynomials in $H[\bm y]$ obtained by taking the sparse resultant of the supports of the $f_i$ and the support of $k=1+\sum y_i x_i$, and plugging in the coefficients of the $f_i$ and $k$.

\begin{theorem}
\label{thm:upper bound by resultant}
Let $l=1+\sum_{i=1}^n h_ix_i$. Then with the notation as above we have 
\[
N^\varphi_{\bm h}(f_1,\ldots, f_n)\leq \mult^\varphi_l(\result{f_1,\ldots, f_n}) .
\]
In particular, we have $N^\varphi_{\bm h}(f_1,\ldots, f_n)\leq \mult_l(\result{f_1,\ldots, f_n})$.
\end{theorem}

\begin{proof}
    Given $g_i\in \varphi^{-1}\{f_i\}$ for $1\leq i\leq n$ with $\bigcap_{i=1}^n V(g_i)$ finite, we have
    \[
    \result{g_1,\ldots, g_n}^\varphi\in \result{f_1,\ldots, f_n} .
    \]
    By Lemma \ref{lem:factors of resultant}, it follows that
    \begin{multline*}
    \left\vert\bigcap_{i=1}^n V(g_i) \cap \varphi^{-1}\{\bm h\} \right\vert= 
    \mult^K_{\varphi^{-1}\{l\}}(\result{g_1,\ldots, g_n})\leq\\
    \leq\mult^K_{\varphi^{-1}\{l\}}(\varphi^{-1}\{\result{f_1,\ldots, f_n}\})
    =\mult^\varphi_l(\result{f_1,\ldots,f_n}) .
    \end{multline*}
\end{proof}

In the remainder of this section, we analyze the utility of Theorem~\ref{thm:upper bound by resultant} in two explicit examples. Our computations rely on the help of the Singular Computer Algebra System \cite{sing}.

\begin{example}
\label{ex:Li Wang counterexample}
Let $a,b,r,s,t$ be positive reals and consider the polynomial system in two variables given by
\[
\begin{cases}
f\coloneqq 1+ax-by = 0\\
g\coloneqq 1+rx^3-sy^3-tx^3y^3=0
\end{cases}
\]
Li and Wang showed that for appropriate choices of $a,b,r,s,t$ the system has $3$ positive real solutions~\cite{LW}. This served as a counterexample to the Itenberg-Roy conjecture that predicted at most $2$ real solutions. We now show that a resultant computation can predict the correct bound. As before, we introduce an auxiliary linear form
\[
l \coloneqq 1+ux+vy
\]
with parameters $u,v$, compute a multiple of the sparse resultant of $l$, $f$, and $g$ and then specialize to the sign hyperfield to obtain a set of signed polynomials in $u$ and $v$. In this set of signed polynomials, some but not all coefficients have a constant sign (up to multiplying everything by $-1$). We use the following Singular code to compute the resultant.

\begin{verbatim}
    system("random", 12341234);
    ring R = (0,(u,v,a,b,r,s,t)),(x,y),dp;
    ideal I = 1+ux+vy, 1+ax-by, 1+rx3-sy3-tx3y3;
    module m = mpresmat(I,0);
    det(m) / b9; // simplify by dividing by b^9
\end{verbatim}

This gives (abbreviating terms with multiple signs)
\begin{multline*}
\textcolor{black!50}{u^{6}(\cdots)}
\textcolor{black!50}{+u^{5}v(\cdots)}
-3u^{5}ab^{3}s
\textcolor{black!50}{+u^{4}v^{2} (\cdots)}
-9u^{4}va^{2}b^{2}s
+3u^{4}a^{2}b^{3}s 
\textcolor{black!50}{+u^{3}v^{3}(\cdots)} \\
\textcolor{black!50}{+u^{3}v^{2} (\cdots)}
\textcolor{black!50}{+u^{3}v(\cdots)}
\textcolor{black!50}{+u^{3}(\cdots)}
\textcolor{black!50}{+u^{2}v^{4}(\cdots)}
\textcolor{black!50}{+u^{2}v^{3}(\cdots)} \\
+u^{2}v^{2}(9a^{4}bs +9ab^{4}r +9abt)
\textcolor{black!50}{+u^{2}v(\cdots)}
+3u^{2}ab^{3}t
\textcolor{black!50}{+uv^{5}(\cdots)}
+9uv^{4}a^{2}b^{2}r \\
\textcolor{black!50}{+uv^{3}(\cdots)}
\textcolor{black!50}{+uv^{2}(\cdots)}
-9uva^{2}b^{2}t
-3ua^{2}b^{3}t
\textcolor{black!50}{+v^{6}(\cdots)}
+3v^{5}a^{3}br
+3v^{4}a^{3}b^{2}r \\
\textcolor{black!50}{+v^{3}(\cdots)}
+3v^{2}a^{3}bt
+3va^{3}b^{2}t
+a^{3}b^{3}t.
\end{multline*}

Specializing to the sign hyperfield, we obtain the set of signed polynomials in $u$ and $v$ represented in Figure~\ref{fig:big resultant}. The maximal boundary multiplicity of the polynomials in this set is $3$, the constaints coming from for the lower boundary.  Since we know that this bound can be achieved by \cite{LW}, the boundary-multiplicity is equal to the multiplicity in this case.
\end{example}

\begin{figure}[htbp]
    \centering
    \begin{tikzpicture}
        \matrix[matrix of math nodes]{
        * &   &   &   &   &   &   \\
        + & * &   &   &   &   &   \\
        + & + & * &   &   &   &   \\
        * & * & * & * &   &   &   \\
        + & * & + & * & * &   &   \\
        + & - & * & * & - & * &   \\
        + & - & + & * & + & - & * \\
        };
    \end{tikzpicture}
    \caption{A multiple of the signed sparse resultant of $f$, $g$ and $l$. A $*$ means the sign is undetermined.}
    \label{fig:big resultant}
\end{figure}

Note that signed resultants are not always the best way to look at certain problems, as the next example shows.

\begin{example}
We compute a multiple of the resultant of $1+ux+vy$, $1+ax+by$ and $1+tx+rx^2-sy^2$ using the following code:
\begin{verbatim}
    system("random", 12341234);
    ring R = (0,(u,v,a,b,r,s,t)),(x,y),dp;
    ideal I = 1+ux+vy, 1+ax+by,1+rx2-sy2+tx; 
    module m = mpresmat(I,0);
    det(m) / b; // simplify by dividing by b
\end{verbatim}
The result is the polynomial in $u$ and $v$ given by
\begin{multline*}
u^2(b^2-s)+uv(-ab+bt)+u(2as-b^2t)+\\
+v^2(a^2-at+r)+v(abt-2br)-a^2s+b^2r .  
\end{multline*}
None of the signs of the coefficients are determined, so our bound is $2$. But clearly $a,b>0$ implies that the system cannot have any positive solutions.
\end{example}

\emergencystretch=3em
\printbibliography
\end{document}